\documentclass[11pt]{article}
\usepackage{color}
\usepackage{amssymb, latexsym, amsmath, epsfig}

\newtheorem{theorem}{Theorem}[section]
\newtheorem{lemma}[theorem]{Lemma}

\newtheorem{definition}[theorem]{Definition}

\newtheorem{algorithm}[theorem]{Algorithm}
\newtheorem{example}[theorem]{Example}

\renewcommand{\theequation}%
{\arabic{section}.\arabic{equation}}

\newcommand\ddelta\bigtriangledown
\newcommand\ld\lambda
\newcommand\Ld\Lambda

\begin{document}


\title{{\Large\textbf{A Posteriori Error Estimates for A Modified Weak Galerkin Finite Element Method  Solving Linear Elasticity Problems}}
\thanks{\footnotesize{Received***
\newline 
Project supported by the National Natural Science Foundation of China(No.11901189), the Natural Science Foundation of Hunan Province(No.\,2022JJ30271).
\newline    $\dag$Corresponding author, E-mail:  Xieyy@m.scnu.edu.cn }}} 

\author{\normalsize{Chunmei Liu$^{1}$, \quad Liuqiang Zhong$^{2}$, \quad Yingying Xie, $^{3,\dag}$, \quad Liping Zhou $^{1}$} %
\\[2mm]
\small{1. College of Science, Hunan University of Science  and Engineering, Yongzhou 425199, China;}  
\\
\small{2. School of Mathematical Sciences, South China Normal University, Guangzhou 510631,  China;}  
\\
\small{3. School of Mathematics and Information Science, GuangZhou University, Guangzhou 510006,  China}  }

\date{}
\maketitle

 {\bf Abstract}\quad
 In this paper,  a residual-type a posteriori error estimator is proposed and analyzed for
 a modified weak Galerkin finite element method solving linear elasticity problems. The estimator is proven to be both reliable and 
efficient because it provides upper and lower bounds on the actual error in a discrete energy norm.
Numerical experiments are given to illustrate the effectiveness of the this error estimator. 


\textbf{Mathematics Subject Classification~~~}65N30, 35J20

 \section{Introduction}

\noindent In this paper, we  consider residual-type  posteriori error estimates for a modified weak Galerkin method  about the following  linear elasticity problems  
\begin{align}\label{ela-Model:1}
 \left \{ 
 \begin{array}{rl}
      -\mu \Delta \pmb{u} - (\lambda+\mu)\nabla(\nabla\cdot \pmb{u}) =\pmb{f}, &  (x,y) \in \Omega, \\
       \pmb{u}=\pmb{g}, & (x,y) \in \partial\Omega,
 \end{array}
\right.
 \end{align}
where $\Omega\subset\mathbb{R}^d (d=2,3)$ is a polytopal domain  with the boundary $\partial\Omega$, 
$\boldsymbol{f}$ si an external force,  $\boldsymbol{u}: \Omega\rightarrow \mathbb{R}^d$ is a displacement vector,  the Lam\'{e} constants $\mu=\frac{E}{2(1+\nu)}$ and    
$\lambda=\frac{E\nu}{(1+\nu)(1-2\nu)}$ which can be composed by the elastic modulus  
$E$  and  Poisson's ratio $\nu\in [0,0.5)$.

In this paper, we assume the solution $\pmb{u}$ of the model\ref{ela-Model:1} satisfies the  $H^2-$ regularity estimate according to \cite{Brenner1992elasticity,Brenner1994Methods}
\begin{eqnarray*}
\| \pmb{u}\|_2+\lambda\|\nabla\cdot\pmb{u}\|_1\leq C\|\pmb{f}\|,
\end{eqnarray*}
where  $C$ is independent of $\lambda$.

The weak Galerkin finite element method(WGFEM) was first prosesed by Wang and Ye \cite{Wang2013problems} to solving a second order elliptic problem, then this method was developed for various PDEs, such as   elliptic problems \cite{Huang2017equations, Li2018problems}, parabolic equations \cite{Zhou2019problems, Zhu2019equation},  Stokes equations\cite{Wang2016equtions,Wang2016stokesequations,WangZhai2018problem}, Navier-Stokes equations\cite{Hu2019equations, Liu2018equations},  Biharmonic equation\cite{Mu2014equation, Zhang2015order}, and so on.  Moreover,  there is also a lot of work in solving linear elastic problems by using the weak finite element method\cite{Chen2016stresses, Harper2019meshes,Wang2016formulation, Wang2018form, Wang2018mixedform, Yi2019elasticity}. In \cite{Chen2016stresses}
and \cite{Wang2018mixedform}, the WGFEMs in mixed form have been developed, although their numerical schemes are different, the solutions of the stress tensors with strong symmetry are achieved;  In \cite{Harper2019meshes} and  \cite{Yi2019elasticity},  the linear elasticity problems are considered on the triangles or tetrahedrons and the tetrahedrons or hexahedrons respectively. But  the standary Raviart-Thomas spaces are used to define the differential of approximate functions,  and  the stable numerical schemes are obtained and the“locking-free” property of the numerical schemes is proved.  In particular, the two numerical schemes don't use stabilizers.  In \cite{Wang2016formulation},  a numerical scheme with“locking-free” property is constructed for mesh generation which are shape regularity;  In \cite{Wang2018form}, the hybrid technique is applied to the WGFEM for the linear elasticity problems, and the optimal error estimates are obtained.

Recently, there have been quite visible research activities on a posteriori error estimates of the WGFEM and the convergence for the adaptive WGFEM for second order elliptic problems \cite{Chen2014problems,Zhang2016problems,Zhang2018problems, 
Adler2019method,Lihengguang2019meshes,  Xie2021problems,
Xie2022problem,Wanghui2022mesh}.  A  residual-type a posteriori error estimator is designed firstly based on the triangle or tetrahedron  meshes in  \cite{Chen2014problems}, and the reliability and efficiency of the estimator are testified.  Then a stabilizer is added to the variational problem of  second order elliptic problems in \cite{Zhang2016problems}, a  residual-type a posteriori error estimator is constructed and the  reliability and efficiency of the estimator are also testified.   A  residual-type a posteriori error estimator is also construced in \cite{Zhang2018problems}, and the form of this estimator is  different from the one in in  \cite{Chen2014problems}. After that, a posteriori error estimator with a simple form is presented in \cite{Lihengguang2019meshes} , and   applied to general meshes such as hybrid, polytopal and meshes with hanging nodes.   
An adaptive algorithm based on WG and modified WG method is designed for the elliptic problem in \cite{Xie2021problems} and \cite{Xie2022problem}, and the convergence of the adaptive algorithm is proved,  respectively.   A  residual-type a posteriori error estimator is designed based on  the weak Galerkin least-squares finite element method applied to the reaction-diffusion equation in \cite{Adler2019method},  the  reliability and efficiency of the estimator are also testified.  A posteriori error estimator of edge residual-type Weak Galerkin
mixed finite element method solving second-order elliptic problems in \cite{Wanghui2022mesh}, where
two different ways of a posteriori error estimator are presented, both of which hold on polygonal mesh. 
The posteriori error estimates of the weak Galerkin method for the Stokes equation have also been studied such as \cite{Zheng2017stokesproblem,Bao2019meshes}. However, to our best knowledge,  there exists no work in the literature about the posteriori error estimates for the linear elasticity problems.
Our work is motivated by the  posteriori error estimates about second order elliptic problems and  the Stokes equations, we design the following the posteriori error estimator for the    linear elasticity problems 
\begin{align*}
\eta^2(\pmb{v}_h,\mathcal{T}_h)&:=\sum\limits_{\tau\in\mathcal{T}_h}\left(\eta_c^2(\pmb{v}_h,\tau)+\eta_{nc}^2(\pmb{v}_h,\tau)+\mathrm{osc}^2(\pmb{f},\tau)+s_{\tau}(\pmb{v}_h, \pmb{v}_h)\right),  
\end{align*} 
where
\begin{align*}
 \eta_c^2(\pmb{v}_h,\tau)&=h_{\tau}^2(\mu^{-1}+(\mu+\lambda)^{-1})\|\pmb{f}+\nabla\cdot(\mu\nabla_w\pmb{v}_h)+\nabla((\mu+\lambda)\nabla_w\cdot\pmb{v}_h)\|^2_{\tau},\\ \eta_{nc}^2(\pmb{v}_h,\tau)&=\mu^{-1}\sum\limits_{e\in \partial \tau} h_e\| J_e(\mu\nabla_w\pmb{v}_h+(\mu+\lambda)(\nabla_w\cdot \pmb{v}_h)\pmb{I})\|^2_e, \\ 
\mathrm{osc}^2(\pmb{f},\tau)&=h^2_\tau(\mu^{-1}+(\mu+\lambda)^{-1})\|\pmb{f}-\pmb{f}_h\|^2_{\tau},\\
s_{\tau}(\pmb{v}_h, \pmb{v}_h)&= h_\tau^{-1}\langle Q_b\pmb{v}_0-\pmb{v}_b, 
Q_b\pmb{v}_0-\pmb{v}_b\rangle_{\partial\tau},
\end{align*}
with  $h_\tau$ being the diameter of the element $\tau$, $h_e$ being the length of edge or face $e$, $\nabla_w\pmb{v}_h$ and $\nabla_w\cdot\pmb{v}_h$ is the weak gradient and the weak divergence of  $\pmb{v}_h$,      $J_e$ represents the jump across the edge or face $e$,   $Q_b$ is the $L^2$ projection operator to $V_{k-1}(e)$,   $\pmb{f}_h$ is the projection of   $\pmb{f}$  to the weak Galerkin finite element space. 

In this work,  we prove the reliability by the following upper  bound 
\begin{align*} 
 \|\mu^{1/2}( \nabla\pmb{u}- \nabla_{w}\pmb{u}_h)\|^2_{\mathcal{T}_h}+
\|(\mu+\lambda)^{1/2} (\nabla\cdot\pmb{u}- \nabla_{w}\cdot\pmb{u}_h)\|^2_{\mathcal{T}_h} \leq  C_1^2\eta^2(\pmb{u}_h, \mathcal{T}_h),
 \end{align*}
and the efficiency by the following  lower bound   
\begin{align*}
\eta^2\lesssim \|\mu^{1/2}( \nabla\pmb{u}- \nabla_{w}\pmb{u}_h)\|^2_{\tau}+
\|(\mu+\lambda)^{1/2} (\nabla\cdot\pmb{u}- \nabla_{w}\cdot\pmb{u}_h)\|^2_{\tau}+ \mathrm{osc}^2(\pmb{f},\mathcal{T}_h).
\end{align*}

In this paper, in addition to a special constant, we always adopt the mark $a\lesssim b$, which
indicates that there is a constant $C$ such that $a\leq C b$.

The rest of this paper is organized as follow. In section 2, we description the spaces of funcitons to be used ,  weak gradient operator, weak divergence operator and present the modified weak finite element scheme.  In section 3, we introduce four modules of adaptive algorithm and the flow of adaptive algorithm.  Section 4 is devoted to the  a posteriori error analysis. In section 5,  we verify the theoretical results by  two numerical examples.

\section{Prelimimaries and Notations}

In order to describe the modified weak Galerkin finite element method, we  recall  the definions of weak gradient and  weak divergence,  the weak Galerkin finite element spaces,  the definions of discrete weak gradient and  discrete weak divergence, and the corresponding  modified weak finite element scheme. 

 For any bounded domain $K\subset \mathbb{R}^d(d=2,3)$ with  Lipschitz continuous boundary $\partial K$ , we use the standard definitions for the Sololev spaces 
\begin{align*}
H^m(K)=\{v\in L^2(K): D^{\alpha} v \in L^2(K), \forall |\alpha | \leq m\},
\end{align*}
where $D^{\alpha} v=\frac{\partial ^{|\alpha|} v}{{\partial x_1}^{\alpha_1}\cdots{\partial x_d} ^{\alpha_d} }$, and $|\alpha|=\alpha_1+\cdots+\alpha_d$. 
Let  $H^m_0(K)$ be  a subspace of  $H^m(K)$ such as
\begin{align*}
H^m_0(K)=\{v\in H^m(K):  v=0~\mathrm{on} ~\partial K\}.
\end{align*}

We  also use the standard definition of norm  $\|\cdot\|_{m,K}$ in these Sobolev spaces $H^m(K)$, $[H^m(K)]^d$ and$[H^m(K)]^{d\times d}$.  Specifically when $m=0$, the space $H^m(K)= L^2(K)$.  In addition,   we denote $<\cdot, \cdot>_{\partial K}$ to be the inner productor duality pairing in  $L^2(\partial K)$, and  $\bold{H}(\mathrm{div}; K)=\{\tau \in [L^2(K)]^{d\times d}: \nabla\cdot \tau \in (L^2(K))^d\}$ with the norm  $\|\tau\|_{\mathrm{div};K}=(\|\tau\|^2_K+\|\nabla\cdot \tau\|^2_K)^{\frac{1}{2}}$. 

\subsection{Weak gradient and divergence operators}
In this subsection, we review the definitions of weak gradient and weak divergence operators which can be applied to descretize the linear elasticity problems \cite{Yi2019elasticity}.  Let  $K$ be any polygonal domain with boundary $\partial K$ and  $e\in\partial K$ be an edge ($d=2$) or a face $d=3$).  Denote  the space of weak vector-valued funciton $\mathcal{V}(K)$  as  follow
\begin{align*}
\mathcal{V}(K)=\left\{\pmb{v}=\{\pmb{v}_0, \pmb{v}_b \}: \pmb{v}_0\in [L^2(K)]^d,  \pmb{v}_b\cdot \pmb{n}\in H^{-\frac{1}{2}}(\partial K) \right\},
\end{align*}
where  $\pmb{n}$ is the unite outward normal vector on $\partial K$,  the first compont $\pmb{v}_0$ and the second component $\pmb{v}_b$ represent   the vector $\pmb{v}$ in  $K$ and on the boundary $\partial K$. Note that  $\pmb{v}_b$ may not necessarily be related to the  trace of 
$\pmb{v}_0$ on $\partial K$, even if the trace is well defined.

According to \cite{Yi2019elasticity},  we describe the definition of the weak divergence as follow.
\begin{definition}(Weak Divergence) 
For any weak vector-valued function $\pmb{v}\in\mathcal{V} (K) $,   the  weak divergence   $\nabla_{w,\tau}\cdot \pmb{v}$ is defined as a linear function in the Sobolev space  $H^1(K)$
\begin{align*}
(\nabla_{w,K}\cdot \pmb{v}, \phi)_K= -(\pmb{v}_0, \nabla \phi)_K+<\pmb{v}_b\cdot \pmb{n}, \phi>_{\partial K} , \forall \phi\in H^1(K).
\end{align*}
\end{definition}

In order to describe the weak gradient operator, we introduce the following  space of weak vector-valued funtions on $K$, such as 
\begin{align*}
\mathcal{W}(K)=\left\{\pmb{v}=\{\pmb{v}_0, \pmb{v}_b \}: \pmb{v}_0\in [L^2(K)]^d,  \pmb{v}_b\in [H^{\frac{1}{2}}(\partial K)]^d\right\}.
\end{align*}

According to \cite{Yi2019elasticity},  we define  the weak gradient as follow.
\begin{definition}(Weak Gradient) \label{def:weakg}
For any weak vector-valued function $\pmb{v}\in\mathcal{W}(K) $,  the weak gradient $\nabla_{w,K}\pmb{v}$  is defined as a linear function in the Sobolev space $ H[(div; K)]^{d\times d}$  
\begin{align*}
(\nabla_{w,K}\pmb{v}, \psi)_\tau= -(\pmb{v}_0, \nabla\cdot\psi)_K+<\pmb{v}_b, \psi \pmb{n}>_{\partial K} , \forall \psi\in [H(\mathrm{div}; K)]^{d},
\end{align*}
$\pmb{n}$ is the unite outward normal vector on $\partial K$.
\end{definition}

\subsection{The modified weak finite element scheme}

In this subsection, we  introduce the modified weak finite element scheme, so some notations are descripted fistly.  Let $\mathcal{T}_h$ be a partition fo the domain $\Omega$  consisiting of elements which are closed and simply connected triangles or tetrahedrons,  let $\mathcal{E}_h$ be the union of all edges or faces of the mesh elements. For any element $\tau\in\mathcal{T}_h$, 
$h_\tau$ denots the diameter of $T$,   $h=\max_{\tau\in \mathcal{T}_h} h_\tau$ denotes the mesh size of  $\mathcal{T}_h$.

For each element $\tau \in \mathcal{T}_h$, let  rigid motion(RM) space be  
\begin{eqnarray*}
RM(\tau)=\{ \pmb{a}+\eta \pmb{x}: \pmb{a}\in  \mathbb{R}^d, \eta\in so(d)\},
\end{eqnarray*}
where $\pmb{x}$ is the position vector on the element $\tau $, $so (d) $ is $d\times d-$dimensional skew-symmetric matrix space. A finite dimension space is formed by traces of functions on each boundary $e\in \partial \tau $in RM space as follows 
\begin{eqnarray*}
P_{RM}(e)=\{\pmb{v}\in [L^2(e)]^d:\pmb{v}=\tilde{\pmb{v}}|_e, \forall \tilde{\pmb{v}}\in RM(\tau), e\subset\partial \tau\}.
\end{eqnarray*}
 
For any integer $k \geq 1 $,  the local weak finite element space on any element $\tau $ is 
\begin{align*}
\mathcal{V}(\tau)=\left\{\pmb{v}=\{\pmb{v}_0, \pmb{v}_b \}: \pmb{v}_0\in [P_k(\tau)]^d,  \pmb{v}_b\in V_{k-1}(e), \forall e \subset \partial\tau \right\},
\end{align*}
wher $V_{k-1}(e)=[P_{k-1}(e)]^d + P_{RM}(e)$, $P_k(\tau)$ is the set of  polynomials ofdegree no greater the $k$  on  $\tau $,  $P_{k-1}(e)$ is the set of polynomials of degree no greater than $k-1$ on   $e\subset\partial \tau$.
Then, we denote the global weak finite element space  $\mathcal{V}_h$  and its subspace  $\mathcal{V}_h^0$ as follows
\begin{align*}
\mathcal{V}_h&=\left\{\pmb{v}=\{\pmb{v}_0, \pmb{v}_b \}: \pmb{v}_0|_{\tau}\in [P_k(\tau)]^d,   
\pmb{v}_b|_{e}\in V_{k-1}(e),  \tau\in\mathcal{T}_h,  \forall e \in  \mathcal{E}_h \right\},\\ 
\mathcal{V}_h^0&=\left\{\pmb{v}=\{\pmb{v}_0, \pmb{v}_b \}\in \mathcal{V}_h: \pmb{v}_b=\pmb{0}~\mathrm{on}~\partial\Omega \right\},
\end{align*}
 then, according to \cite{Brenner1994Methods, Zhangran2020problems}, we denote the local matrix-valued function space  $\Sigma_h(\tau)$  and the global matrix-valued function space $\Sigma_h$ as follows 
\begin{align*}
\Sigma(\tau)&=\left\{\pmb{w}\in[P_{k-1}(\tau)]^{d\times d}\right\}. \\
\Sigma_h&=\left\{\pmb{w}\in[L^2(\tau)]^{d\times d}: \pmb{w}|_{\tau}\in \Sigma(\tau)\right\}.
\end{align*}

Based on these definitions above,  we now introduc the discrete weak gradient operator, the discrete weak divergence operator and the modified weak finite element scheme.
\begin{definition}(Discrete Weak Divergence,\cite{Yi2019elasticity,Zhangran2020problems}) 
For each  $\tau\in \mathcal{T}_h$,  the discret  weak divergenc $\nabla_{w,\tau}\cdot \pmb{v}\in P_{k-1}(\tau)$  of $\pmb{v}\in\mathcal{V} (\tau) $ satisfies the following formula
\begin{align} \label{discon-weak-diver}
(\nabla_{w,\tau}\cdot \pmb{v}, \phi)_\tau= -(\pmb{v}_0, \nabla \phi)_\tau+<\pmb{v}_b\cdot \pmb{n}, \phi>_{\partial\tau} , \forall \phi\in P_{k-1}(\tau),
\end{align}
where  $\pmb{n}$  is the unite outward normal vector on $\partial \tau$ .
\end{definition}

\begin{definition}(Discrete Weak Gradient,\cite{Yi2019elasticity,Zhangran2020problems}) 
For each  $\tau\in \mathcal{T}_h$,  the discret weak gradient $\nabla_{w,\tau} \pmb{v}\in \Sigma(\tau)$  of  $\pmb{v}\in\mathcal{V}(\tau) $, satisfies the following formula
\begin{align}\label{discon-weak-grad}
(\nabla_{w,\tau}\pmb{v}, \psi)_\tau= -(\pmb{v}_0, \nabla\cdot\psi)_\tau+<\pmb{v}_b, \psi \pmb{n}>_{\partial\tau} , \forall \psi\in\Sigma(\tau),
\end{align}
 where  $\pmb{n}$  is the unite outward normal vector on $\partial \tau$ .
\end{definition}

Now, we turn to present the modified weak finite element scheme.   Using the finite element space of order $k$, we introduce the following discrete variational 
problem of (\ref{ela-Model:1}):  Find $\pmb{u}_h=\{\pmb{u}_0,\pmb{u}_b\}\in \mathcal{V}_h$, $\pmb{u}_b|_{\partial \Omega}=Q_b{\pmb{g}}$, such that
\begin{align}\label{Dis-Weak-Gal}
 a_w(\pmb{u}_h,\pmb{v})=(\pmb{f},\pmb{v}_0), \forall  \pmb{v}=\{\pmb{v}_0,\pmb{v}_b\}\in \mathcal{V}_h^0,
 \end{align}
where  the bilinear form $a_w(\cdot,\cdot)$ is defined by
\begin{align}\label{discon-weak-problem}
 a_w(\pmb{w},\pmb{v})=\mu\sum\limits_{\tau\in\mathcal{T}_h}(\nabla_w\pmb{w},\nabla_w\pmb{v})_\tau+
(\mu+\lambda)\sum\limits_{\tau\in\mathcal{T}_h}(\nabla_w\cdot\pmb{w},\nabla_w\cdot\pmb{v})_\tau +s(\pmb{w},\pmb{v}),
 \end{align}
and 
\begin{align}\label{s}
 s(\pmb{w},\pmb{v})=\sum_{\tau\in \mathcal{T}_h} h_\tau^{-1}\langle Q_b\pmb{w}_0-\pmb{w}_b, 
Q_b\pmb{v}_0-\pmb{v}_b\rangle_{\partial\tau},
 \end{align} 
here,  $Q_b$ is the local $L^2-$ projection onto the space $V_{k-1}(e)$. 
 
According to \cite{Yi2019elasticity,Zhangran2020problems},  there exists a unique solution to the modifed weak Galerkin finite element method defined in  (\ref{discon-weak-problem}).  
In our paper, we will not repeat this conclusion. 

In the following section, we  give a brief introduction of the adaptive algorithm based on the modified  weak Galerkin finite element method  by refering to the standard AFEM in \cite{Cascon2008method,Liu2016dimensions}.

\section{An adaptive modified weak Galerkin algorithm}
\setcounter{equation}{0}

Let $\mathcal{T}_0$ be a
triangles grid  or tetrahedrons grid on the bounded domain $Ω$, and let $\{\mathcal{T}_l\}_{l>0}$ be a sequence of nested grids by a series of local refinement. The grid $\mathcal{T}_{l+1}$ is generated from $\mathcal{T}_{l}$ by the following four algorithm modules :
 
\begin{equation}\label{equation_8}
{\mbox{\bf SOLVE}}\ \rightarrow\ {\mbox{\bf ESTIMATE}}\ \rightarrow\
{\mbox{\bf MARK}}\ \rightarrow \ {\mbox{\bf REFINE}}.
\end{equation}

The specific roles of these four modules are as follows:

(1)\mbox{\bf SOLVE}

For the given functions $\pmb{f}\in (L^2(\Omega))^d(d=2,3)$ and a given grid $\mathcal{T}_l$, we assume that the algorithm module {\bf SOLVE} exactly outputs the discrete solution $\boldsymbol{u}_l$  of (\ref{discon-weak-problem}) as
 \begin{eqnarray*}
\boldsymbol{u}_l={\bf SOLVE}(\mathcal{T}_l, \pmb{f},  \pmb{g}) \in\mathcal{V}_l.
\end{eqnarray*}

(2) {\bf ESTIMATE}  

For  a given grid $\mathcal{T}_l$ ,  let  $e\in \mathcal{E}_l$ be shared by two  element $\tau_1$ and $\tau_2$,let  $\pmb{n}_1$ and $\pmb{n}_2$ be  the unite outward normal vector on $e$ belong to  $\tau_1$ and $\tau_2$ respectively.   For any $\pmb{w}\in \Sigma_h$,  denote $[\pmb{w}]_e= \pmb{w}|_{\tau_1}- \pmb{w}|_{\tau_2}$,  $[\pmb{w}\pmb{n}]_e=\pmb{w}|_{\tau_1}\pmb{n}_1+\pmb{w}|_{\tau_2}\pmb{n}_2$. 
 
We denote the jump across  $e$  as follows 
\begin{align*}
J_e(\mu\nabla_w\pmb{v}_l+(\mu+\lambda)(\nabla_w\cdot \pmb{v}_l)\pmb{I})&=
\begin{cases}
[(\mu\nabla_w\pmb{v}_l+(\mu+\lambda)(\nabla_w\cdot \pmb{v}_l)\pmb{I})\pmb{n}]_e,&   \mbox{if}~e\in \mathcal{E}_l^{0},\\
0, &  \mbox{otherwise},
\end{cases}\\
J_e^0(\pmb{v}^l_0)&=
\begin{cases}
\pmb{v}^l_0 |_{\partial \tau_1}-\pmb{v}^l_0 |_{\partial \tau_2},  & \mbox{if}~e\in \mathcal{E}_l^{0},\\
0, &  \mbox{otherwise}.
\end{cases}
\end{align*}

For  a given grid $\mathcal{T}_l$ and a given function $\pmb{u}_k\in \mathcal{V}_l$, the posteriori error estimator based on $\tau$ is given by
\begin{align}\label{eta:1}
\eta^2(\pmb{u}_l,\mathcal{T}_l)&:=\sum\limits_{\tau\in\mathcal{T}_l}\left(\eta_c^2(\pmb{u}_h,\tau)+\eta_{nc}^2(\pmb{u}_h,\tau)+\mathrm{osc}^2(\pmb{f},\tau)+s_{\tau}(\pmb{u}_h, \pmb{u}_h)\right),  
\end{align} 
where
\begin{align*}
 \eta_c^2(\pmb{u}_l,\tau)&=h_{\tau}^2(\mu^{-1}+(\mu+\lambda)^{-1})\|\pmb{f}+\nabla\cdot(\mu\nabla_w\pmb{u}_l)+\nabla((\mu+\lambda)\nabla_w\cdot\pmb{u}_l)\|^2_{\tau},\\ \eta_{nc}^2(\pmb{u}_h,\tau)&=\mu^{-1}\sum\limits_{e\in \partial \tau} h_e\| J_e(\mu\nabla_w\pmb{u}_l+(\mu+\lambda)(\nabla_w\cdot \pmb{u}_l)\pmb{I})\|^2_e, \\ 
\mathrm{osc}^2(\pmb{f},\tau)&=h^2_\tau(\mu^{-1}+(\mu+\lambda)^{-1})\|\pmb{f}-\pmb{f}_l\|^2_{\tau},\\
s_{\tau}(\pmb{u}_l, \pmb{u}_l)&= h_\tau^{-1}\langle Q_b\pmb{u}_{l,0}-\pmb{u}_{l,b}, 
Q_b\pmb{u}_{l,0}-\pmb{u}_{l,b}\rangle_{\partial\tau},
\end{align*}
with  $h_\tau$ being the diameter of the element $\tau$, $h_e$ being the length of edge or face $e$, $\nabla_w\pmb{u}_l$ and $\nabla_w\cdot\pmb{u}_l$ is the weak gradient and the weak divergence of  $\pmb{u}_l$,      $J_e$ represents the jump across the edge or face $e$,   $Q_b$ is the $L^2$ projection operator to $V_{l-1}(e)$,   $\pmb{f}_l$ is the projection of   $\pmb{f}$  to the weak Galerkin finite element space. 

For any $\mathcal{W}_l\subset \mathcal{T}_l$ and $\pmb{u}_k\in \mathcal{V}_l$,  define the following sets by
\begin{align*}
 \eta^2(\pmb{u}_l,\mathcal{W}_l)=\sum\limits_{\tau\in\mathcal{W}_l}\eta_c^2(\pmb{u}_l,\tau), \ 
\mathrm{osc}^2(\pmb{f},\mathcal{W}_l)=\sum\limits_{\tau\in\mathcal{W}_l}\mathrm{osc}^2(\pmb{f},\tau).
\end{align*}

For any given grid $\mathcal{T}_l$ and the corresponding discrete exact $\pmb{u}_l\in \mathcal{V}_l$ of (\ref{discon-weak-problem}),  we can obtain the  posteriori error estimator $\eta^2_{\mathcal{T}_l}(\boldsymbol{u}_l,\tau)$ of any element $\tau \in\mathcal{T}_l$ by the following algorithm module
\begin{align*}
\eta^2_{\mathcal{T}_l}(\boldsymbol{u}_l,\tau)=\mathrm{{\bf ESTIMATE}}(\mathcal{T}_l,\boldsymbol{u}_l,\boldsymbol{f},  \boldsymbol{g}).
\end{align*}

(3) {\bf MARK}

In this paper, we utilize the D\"{o}rfler marking way(\cite{Dorfler1966equation}) to mark elements which will be refined.
Given a grid $\mathcal{T}_l$, a set of posteriori error  estimators $\{\eta^2_{\mathcal{T}_l}(\boldsymbol{u}_l,\tau)\}_{\tau\in
\mathcal{T}_l}$  and a D\"{o}rfler marking  parameter $\vartheta\in (0, 1)$, we can get a
marked element set $\mathcal{M}_l\subset \mathcal{T}_l$ by the following  algorithm module
\begin{eqnarray*}
\mathcal{M}_l=\mathrm{{\bf MARK}}(\eta^2(\boldsymbol{u}_l,\mathcal{T}_l),
\mathcal{T}_l, \vartheta),
\end{eqnarray*}
in addition, the set $\mathcal{M}_l$ satisfies
\begin{eqnarray*}
\eta^2_{\mathcal{T}_l}(\boldsymbol{u}_l,\mathcal{M}_l)\geq \vartheta
\eta^2(\boldsymbol{u}_l,\mathcal{T}_l)
\end{eqnarray*}
and has a minimal cardinality.

(4)   {\bf REFINE}

We assume that a module {\bf REFINE} implements an iterative or a recursive bisection (see\cite{Xie2021problems}).   For a given
number $l>1$, any grid $\mathcal{T}_k\in  \mathcal{L}(\mathcal{T}_0)$  and a subset $\mathcal{M}_l\subset \mathcal{T}_l$, we can obtain a conforming grid $\mathcal{T}_{l+1}\in  \mathcal{L}(\mathcal{T}_0)$ by the algorithm module {\bf REFINE} as
\begin{align*}
\mathcal{T}_{l+1}= \mathrm{{\bf REFINE}}(\mathcal{T}_{l},\mathcal{M}_{l}).
\end{align*}

 Using the above four algorithm modules, we design an adaptive modified weak Galerkin finite element method(AMWG-FEM) as follow. 
\begin{algorithm}[AMG-FEM]\label{ela-AFEM}
 For given functions $\boldsymbol{f}$, $\boldsymbol{g}$, choosing a D\"{o}rfler marking parameter $\vartheta\in (0, 1)$ and a error control constant $tol$, the modules of AMWG-FEM algorithm  is
\begin{enumerate}
  \item Give an initial conforming grid $\mathcal{T}_0$ and set $l=0$.
  \item $\boldsymbol{u}_l=\mathrm{{\bf SOLVE}}(\mathcal{T}_l, \boldsymbol{f}, \boldsymbol{g})$.
  \item $\eta^2_{\mathcal{T}_l}(\boldsymbol{u}_l,\tau)=\mathrm{{\bf ESTIMATE}}(\mathcal{T}_l,\boldsymbol{u}_l,
\boldsymbol{f}, \boldsymbol{g})$.  If $\eta^2(\boldsymbol{u}_l,  {\mathcal{T}_l}) <tol$, then the algorithm stops.
  \item $\mathcal{M}_l=\mathrm{{\bf MARK}}(\eta^2(\boldsymbol{u}_l,\mathcal{T}_l),
\mathcal{T}_l, \vartheta)$.
  \item $\mathcal{T}_{l+1}=\mathrm{{\bf REFINE}}(\mathcal{T}_l,
  \mathcal{M}_l)$.
  \item Set $l = l + 1$ and go to 2.
\end{enumerate}
\end{algorithm}

\section{A posteriori error analysis for the MWG method}
 \setcounter{equation}{0}

This section is devoted to  a study of reliability and  efficiency  for the error estimator
 $\eta(\pmb{u}_h,\mathcal{T}_k)$
defined in (\ref{eta:1}).  Firstly,  we give the followinn three lemmas.
\begin{lemma}\label{lem:weakg-s}
For any  $\pmb{v}=\{\pmb{v}_0, \pmb{v}_b\}\in \mathcal{V}_h$,  we have 
\begin{eqnarray}\label{equ:operator1}
\|\nabla_w\pmb{v}-\nabla\pmb{v}_0\|_{\mathcal{T}_h}^2
\lesssim s(\pmb{v}, \pmb{v}).
\end{eqnarray}
\end{lemma}
\begin{proof}
By the definition \ref{def:weakg} and  Green formula, we will get the relationship between the weak gradient and classical gradient as follows
\begin{align*}
(\nabla_{w,\tau}\pmb{v}, \boldsymbol{\varphi})_\tau= (\nabla \pmb{v}_0,\boldsymbol{\varphi})_\tau - \langle \pmb{v}_0-\pmb{v}_b, \boldsymbol{\varphi} \pmb{n}\rangle_{\partial\tau} , \forall \boldsymbol{\varphi}\in [H^1(\mathrm{div};\tau)]^{d}.
\end{align*}

Let $\boldsymbol{\varphi}=\nabla_w\pmb{v}-\nabla\pmb{v}_0$,  by using trace inequality, we have
\begin{eqnarray*}
\lefteqn{\|\nabla_w\pmb{v}-\nabla\pmb{v}_0\|_{\mathcal{T}_h}^2}\\
&&=\sum_{\tau\in\mathcal{T}_h}\langle \pmb{v}_0-\pmb{v}_b, (\nabla_{w,\tau}\pmb{v}-\nabla\pmb{v}_0)\pmb{n}\rangle_{\partial \tau}
\\
&&\leqslant (s(\pmb{v}, \pmb{v}))^{1/2}\cdot \sum_{\tau\in\mathcal{T}_h} h_\tau^{1/2}\|\nabla_{w,\tau}\pmb{v}_h-\nabla\pmb{v}_0\|_{\partial \tau}\\
&&\lesssim  (s(\pmb{v}, \pmb{v}))^{1/2}\cdot \|\nabla_w\pmb{v}-\nabla\pmb{v}_0\|_{\mathcal{T}_h}, 
\end{eqnarray*}

Dividing $\|\nabla_w\pmb{v}-\nabla\pmb{v}_0\|_{\mathcal{T}_h}$ on both sides of the above equation, then  we obtain (\ref{lem:weakg-s}). 
\end{proof}
 
 \begin{lemma}\label{lem:J_3s}
For any $\pmb{v}=\{\pmb{v}_0, \pmb{v}_b\}\in \mathcal{V}_h$, we have
	\begin{eqnarray}\label{equ:J_3s}
	\sum_{e\in\mathcal{E}_h}h_e^{-1}\|[\pmb{v}_0]\|_e^2\lesssim s(\pmb{v}, \pmb{v}).
	\end{eqnarray}	
\end{lemma}
\begin{proof}
Noting that  $\|[\pmb{v}_b]\|_e=0, \forall e\in \mathcal{E}_h$, we obtain
\begin{eqnarray}\label{equ:J_3s-1}
\sum_{e\in\mathcal{E}_h}h_e^{-1}\|[\pmb{v}_0]\|_e^2
=\sum_{e\in\mathcal{E}_h}h_e^{-1}\|[\pmb{v}_0-\pmb{v}_b]\|_e^2,
\end{eqnarray}	

For any edge or face  $e\in\mathcal{E}_h^0$, there exists $\tau_1\in \mathcal{T}_h$ and $\tau_2\in \mathcal{T}_h$, such that $e=\partial \tau_1\cap\partial\tau_2$. By using of  Cauchy-Schwarz inequality, we get
\begin{eqnarray*}
\lefteqn{\|[ \pmb{v}_0-\pmb{v}_b] \|_e^2}\\
&&=\langle[\pmb{v}_0-\pmb{v}_b],\pmb{v}_0-\pmb{v}_b\rangle_{\partial \tau_1\cap e} + \langle[\pmb{v}_0-\pmb{v}_b],\pmb{v}_0-\pmb{v}_b\rangle_{\partial \tau_2\cap e}\\
&&\leqslant \|[ \pmb{v}_0-\pmb{v}_b]  \|_e(\| \pmb{v}_0-\pmb{v}_b \|_{\partial \tau_1\cap e} +\| \pmb{v}_0-\pmb{v}_b\|_{\partial \tau_2\cap e});
\end{eqnarray*}	
 
Similarly,   for any boundary edge or face $e\in\mathcal{E}_h^\partial$,  
a similar conclusion can be proved . We now sum over  $e\in\mathcal{E}_h$ and the following estimate is true 
\begin{eqnarray*}
\sum_{e\in\mathcal{E}_h} h_e^{-1}\|[\pmb{v}_0-\pmb{v}_b] \|_e^2\lesssim \sum_{e\in\mathcal{E}_h} h_e^{-1/2}\|[\pmb{v}_0-\pmb{v}_b] \|_e \cdot (s(\pmb{v}, \pmb{v}))^{1/2},
\end{eqnarray*}
that is
\begin{eqnarray}\label{equ:J_3s-2}
\sum_{e\in\mathcal{E}_h} h_e^{-1}\|[\pmb{v}_0-\pmb{v}_b] \|_e^2\lesssim s(\pmb{v}, \pmb{v}).
\end{eqnarray}

Combine \eqref{equ:J_3s-1} and \eqref{equ:J_3s-2}, we obtain the conclusion (\ref{lem:J_3s}).
\end{proof}
 
Let $\tilde{\mathcal{V}}_{h}=\{\pmb{v}, \pmb{v}\in [P_0(\tau)]^d, \forall \tau\in \mathcal{T}_h\}$, $\mathcal{V}^c_h=  (H_0^1(\Omega))^d\cap (P_1(\tau))^2$,  we introduce the following estimats by  refer to \cite{BonitoNochetto10:734}.
\begin{lemma}\label{lem:Interpolation}
For any $\tau\in \mathcal{T}_h$, there exists an interpolation operator $I_{\mathcal{T}_h}^c: \tilde{\mathcal{V}}_{h}\rightarrow \mathcal{V}^c_h$, such that  
\begin{equation}\label{InterpolationEstimate1}
\|v_{\mathcal{T}_h} - I_{\mathcal{T}}^c v_{\mathcal{T}_h}\|_{\mathcal{T}_h}\lesssim h\|\nabla v_{\mathcal{T}_h}\|_{L^2(\Omega)}, \forall v_{\mathcal{T}_h}\in H_0^1(\Omega),
\end{equation}
where the constant is  only dependent on the shape regular of mesh $\mathcal{T}_h$. 
For any $|a| = 0, 1$,  we have
\begin{equation}\label{InterpolationEstimate}
\|D^a(v_{\mathcal{T}} - I_{\mathcal{T}}^c v_{\mathcal{T}_h})\|_{\mathcal{T}}^2\lesssim \sum_{e\in \mathcal{E}_\mathcal{T}}h_\tau^{1-2|a|}\|[ v_{\mathcal{T}}]_e\|^2_e, \forall v_{\mathcal{T}}\in \tilde{\mathcal{V}}_{h}.
\end{equation}
where the constant is  only dependent on the shape regular of mesh $\mathcal{T}_h$.
\end{lemma} 
 
Now, we shall present the reliability for the error estimator defined in  (\ref{eta:1}) by the following upper bound estimate.
\begin{theorem}\label{lem:7}
Let  $\pmb{u}$  be the solution of (\ref{ela-Model:1}) and  $\pmb{u}_h=\{\pmb{u}_0^h,\pmb{u}_b^h\}\in  \mathcal{V}_h$ be the solution of (\ref{Dis-Weak-Gal}), respectively.  There exists a constant $C_1>0$, such that such 
\begin{align}\label{lemma7_1} 
 \|\mu^{1/2}( \nabla\pmb{u}- \nabla_{w}\pmb{u}_h)\|^2_{\mathcal{T}_h}+
\|(\mu+\lambda)^{1/2} (\nabla\cdot\pmb{u}- \nabla_{w}\cdot\pmb{u}_h)\|^2_{\mathcal{T}_h} \leq  C_1^2\eta^2(\pmb{u}_h, \mathcal{T}_h),
 \end{align}
when the constant $C_1>0$, only depends on the shape regularity of $\mathcal{T}_h$.
\end{theorem}

\begin{proof}
Let $\pmb{e}_1=\mu(\nabla \pmb{u}-\nabla_w\pmb{u}_h)$, $e_2=(\mu+\lambda)(\nabla\cdot \pmb{u}-\nabla_w\cdot\pmb{u}_h)$.  By Lemma  \ref{lem:Interpolation},  we have  $\pmb{u}_h^c=I_{\mathcal{T}_h}^c\pmb{u}_0^h\in\mathcal{V}^c(\mathcal{T}_h)$ and
 \begin{align}\label{lemma7_2} 
&\nabla_w\pmb{u}_h^c=\nabla\pmb{u}_h^c,  \\\label{lemma7_3} 
 &\nabla_w\cdot\pmb{u}_h^c=\nabla\cdot\pmb{u}_h^c.
\end{align}

By using the above notations, the following estimate  is ture 
\begin{align}\nonumber
 E_h^2&=\|\mu^{1/2}( \nabla\pmb{u}- \nabla_{w}\pmb{u}_h)\|^2_{\tau}+
\|(\mu+\lambda)^{1/2} (\nabla\cdot\pmb{u}- \nabla_{w}\cdot\pmb{u}_h)\|^2_{\tau}\\ \nonumber
&=(\pmb{e}_1, \nabla\pmb{u}- \nabla_{w}\pmb{u}_h)_{\tau}+ 
(e_2, \nabla\cdot\pmb{u}- \nabla_{w}\cdot\pmb{u}_h)_{\tau}  \\\nonumber
&=(\pmb{e}_1, \nabla\pmb{u}- \nabla\pmb{u}_h^c)_{\tau}
+ (\pmb{e}_1, \nabla\pmb{u}_h^c- \nabla_{w}\pmb{u}_h)_{\tau}
+(e_2, \nabla\cdot\pmb{u}- \nabla\cdot\pmb{u}_h^c)_{\tau} 
+(e_2, \nabla\cdot\pmb{u}_h^c- \nabla_{w}\cdot\pmb{u}_h)_{\tau}  \\\label{lemma7_3} 
&=I_1+I_2,
 \end{align}
where $I_1=(\pmb{e}_1, \nabla\pmb{u}- \nabla\pmb{u}_h^c)_{\tau}+(e_2, \nabla\cdot\pmb{u}- \nabla\cdot\pmb{u}_h^c)_{\tau} $, $I_2=(\pmb{e}_1, \nabla\pmb{u}_h^c- \nabla_{w}\pmb{u}_h)_{\tau}+(e_2, \nabla\cdot\pmb{u}_h^c- \nabla_{w}\cdot\pmb{u}_h)_{\tau}$.

Firstly, we shall estimate $I_1$.
Let $\pmb{w}=\pmb{u}-\pmb{u}_h^c\in (H_0^1(\Omega))^2$, by refer to\cite{Chen2014problems}, we know that there exists an interpolation operator  $\pmb{w}_h$ which satisfies
 \begin{align}\label{lemma7_4} 
(\pmb{e}_1, \nabla\pmb{w}_h)_{\tau}+(e_2, \nabla\cdot\pmb{w}_h)_{\tau}=\pmb{0}.
 \end{align}
 
 Using (\ref{lemma7_4}), Green formula, the continuity of $\nabla\pmb{u}$ and $\nabla\cdot\pmb{u}$ on the edge or face of the unit $\tau$,  the estimate in \cite{Chen2014problems},  we obtain
 \begin{align}\nonumber
 I_1&=(\pmb{e}_1, \nabla\pmb{u}- \nabla\pmb{u}_h^c)_{\tau}+(e_2, \nabla\cdot\pmb{u}- \nabla_{w}\cdot\pmb{u}_h^c)_{\tau}\\ \nonumber
 &=(\pmb{e}_1, \nabla\pmb{w})_{\tau}+(e_2, \nabla\cdot\pmb{w})_{\tau}\\ \nonumber
&=(\pmb{e}_1, \nabla(\pmb{w}-\pmb{w}_h))_{\tau}+(e_2, \nabla\cdot(\pmb{w}-\pmb{w}_h))_{\tau}\\ \nonumber
&=-(\nabla\cdot\pmb{e}_1, \pmb{w}-\pmb{w}_h)_{\tau}+(\pmb{e}_1 \pmb{n}, \pmb{w}-\pmb{w}_h)_{\partial\tau}-(\nabla e_2, \pmb{w}-\pmb{w}_h)_{\tau}+(e_2, (\pmb{w}-\pmb{w}_h)\pmb{n})_{\partial\tau}\\ \nonumber
&=-(\nabla\cdot\pmb{e}_1+\nabla e_2, \pmb{w}-\pmb{w}_h)_{\tau}+(\pmb{e}_1\pmb{n}, \pmb{w}-\pmb{w}_h)_{\partial\tau}+(e_2, (\pmb{w}-\pmb{w}_h)\cdot\pmb{n})_{\partial\tau}\\ \nonumber
&=(\pmb{f}+\nabla\cdot(\mu\nabla_w\pmb{u}_h)+\nabla((\mu+\lambda)\nabla_w\cdot\pmb{u}_h), \pmb{w}-\pmb{w}_h)_{\tau}+(\mu\nabla_w\pmb{u}_h\pmb{n}, \pmb{w}-\pmb{w}_h)_{\partial\tau}\\ \nonumber
&~~~+(((\mu+\lambda)\nabla_w\cdot\pmb{u}_h)\pmb{n}, \pmb{w}-\pmb{w}_h)_{\partial\tau}\\ \nonumber
&=(\pmb{f}+\nabla\cdot(\mu\nabla_w\pmb{u}_h+(\mu+\lambda)\nabla_w\cdot\pmb{u}_h\pmb{I}), \pmb{w}-\pmb{w}_h)_{\tau}\\ \nonumber
&~~~+((\mu\nabla_w\pmb{u}_h +(\mu+\lambda)\nabla_w\cdot\pmb{u}_h\pmb{I})\pmb{n}, \pmb{w}-\pmb{w}_h)_{\partial\tau}\\ \nonumber
&\lesssim\|\pmb{f}+\nabla\cdot(\mu\nabla_w\pmb{u}_h+(\mu+\lambda)\nabla_w\cdot\pmb{u}_h\pmb{I})\|_{\tau}
\|\pmb{w}-\pmb{w}_h\|_{\tau}  \\ \nonumber
&~~~+\|[(\mu\nabla_w\pmb{u}_h +(\mu+\lambda)\nabla_w\cdot\pmb{u}_h\pmb{I})\pmb{n}]\|_{\partial\tau}\|\pmb{w}-\pmb{w}_h\|_{\partial\tau}\\ \nonumber
&\lesssim\|\pmb{f}+\nabla\cdot(\mu\nabla_w\pmb{u}_h+(\mu+\lambda)\nabla_w\cdot\pmb{u}_h\pmb{I})\|_{\tau}h_{\tau}\|\nabla_w \pmb{w} \|_{\tau}\\\nonumber 
&~~~+(\|[(\mu\nabla_w\pmb{u}_h +(\mu+\lambda)\nabla_w\cdot\pmb{u}_h\pmb{I})\pmb{n}]\|_{\partial\tau}) h_{e}^{\frac{1}{2}}\|\nabla_w  \pmb{w}\|_{\tau}\\   \label{lemma7_5} 
&\lesssim  \eta(\pmb{u}_h, \mathcal{T}_h)\|\nabla\pmb{u}-\nabla\pmb{u}_h^c\|_{\tau}.
 \end{align}
 
By using the Lemma  \ref{lem:weakg-s} and  Lemma\ref{lem:J_3s}, we shall estimate $\|\nabla\pmb{u}-\nabla\pmb{u}_h^c\|_{\tau}$.
\begin{align}\nonumber
\|\nabla\pmb{u}-\nabla\pmb{u}_h^c\|_{\tau}&\leq \|\nabla\pmb{u}-\nabla_w\pmb{u}_h\|_{\tau}
+\|\nabla_w\pmb{u}_h-\nabla_w\pmb{u}_0^h\|_{\tau}+\|\nabla_w\pmb{u}_0^h-\nabla\pmb{u}_h^c\|_{\tau}\\ \nonumber
&\leq \|\nabla\pmb{u}-\nabla_w\pmb{u}_h\|_{\tau}
+\|\nabla_w\pmb{u}_h-\nabla\pmb{u}_0^h\|_{\tau}+\|\nabla\pmb{u}_0^h-\nabla\pmb{u}_h^c\|_{\tau}\\ \nonumber
&\lesssim \|\nabla\pmb{u}-\nabla_w\pmb{u}_h\|_{\tau}
+s(\pmb{u}_h, \pmb{u}_h)+h_e^{-\frac{1}{2}}\|[\pmb{u}_0^h]\|_{e}\\   \label{lemma_eta_7_5} 
& \lesssim E_h+s(\pmb{u}_h, \pmb{u}_h)\lesssim E_h+ \eta^2(\pmb{u}_h, \mathcal{T}_h),
\end{align}
 
Combine (\ref{lemma7_5} ) and (\ref{lemma_eta_7_5} ), we obtain
\begin{align}  \label{lemma7_6}
I_1\lesssim  \eta(\pmb{u}_h, \mathcal{T}_h) E_h + \eta^2(\pmb{u}_h, \mathcal{T}_h).
\end{align}

Secondly, we shall estimate   $I_2$.  By using the Lemma \ref{lem:weakg-s} and Lemma \ref{lem:J_3s}, the relationship between the weak gradient and classical gradient, the norm of the gradient is less than the norm of divergence, the relationship between the weak divergence and classical divergence, we get
\begin{align}\nonumber
I_2&=(\pmb{e}_1, \nabla\pmb{u}_h^c- \nabla_{w}\pmb{u}_h)_{\tau}+(e_2, \nabla\cdot\pmb{u}_h^c- \nabla_{w}\cdot\pmb{u}_h)_{\tau} \\ \nonumber
&=(\pmb{e}_1, \nabla\pmb{u}_h^c- \nabla\pmb{u}_0^h)_{\tau}+(\pmb{e}_1, \nabla\pmb{u}_0^h- \nabla_{w}\pmb{u}_h)_{\tau}\\ \nonumber
&+(e_2, \nabla\cdot\pmb{u}_h^c- \nabla\cdot\pmb{u}_0^h)_{\tau}+(e_2, \nabla\cdot\pmb{u}_0^h- \nabla_{w}\cdot\pmb{u}_h)_{\tau} \\ \nonumber
&
\leq \|\pmb{e}_1\|_{\tau}s(\pmb{u}_h, \pmb{u}_h)+\|\pmb{e}_1\|_{\tau}s(\pmb{u}_h, \pmb{u}_h)+\|e_2\|_{\tau}\|\nabla\pmb{u}_h^c-\nabla_w\pmb{u}_0^h\|_{\tau}+\|e_2\|_{\tau}s(\pmb{u}_h, \pmb{u}_h) \\  \label{lemma7_7}
&\lesssim E_h s(\pmb{u}_h, \pmb{u}_h) \leq E_h \eta(\pmb{u}_h, \mathcal{T}_h).
\end{align}

Combine (\ref{lemma7_6}) and (\ref{lemma7_7}), we have 
\begin{align}\nonumber
E_h^2=I_1+I_2  \lesssim  E_h \eta(\pmb{u}_h, \mathcal{T}_h)+\eta^2(\pmb{u}_h, \mathcal{T}_h)+E_h \eta(\pmb{u}_h, \mathcal{T}_h).
\end{align}

Using the inequality $2ab\leq \varepsilon a^2+\frac{1}{\varepsilon}b^2$, mergeing items with the same form,  we  have completed the proof. 

\end{proof}


Next,  we shall use the standard bubble function technique to prove the efficiency estimate(see \cite{Chen2014problems}). Let  $\omega_e=\tau_1\cup \tau_2$, where $\tau_1$ and $\tau_2$ share the edge or face $e$. We present the following lemma.
\begin{lemma}
There exists a constant $C>0$,  such that
\begin{align}\nonumber
&h_{\tau}\|\pmb{f}+\nabla\cdot(\mu\nabla_w\pmb{v}_h)+\nabla((\mu+\lambda)\nabla_w\cdot\pmb{v}_h)\|_{\tau}\\ \label{LocalLowerBou:1}
&\leq C(
\|\mu^{1/2}( \nabla\pmb{u}- \nabla_w\pmb{u}_h)\|_{\tau}+
\|(\mu+\lambda)^{1/2} (\nabla\cdot\pmb{u}- \nabla_w\cdot\pmb{u}_h)\|_{\tau}+h_{\tau}\|\pmb{f}-\pmb{f}_h\|_{\tau}).\\  \nonumber
&h_e^{1/2}\|[(\mu \nabla_w\pmb{v}_h+(\mu+\lambda)(\nabla_w\cdot\pmb{v}_h )\pmb{I})\pmb{n}]\|_e\\ \label{LocalLowerBou:2} 
&\leq C 
 \left(\mathrm{osc}(\pmb{f}, \omega_e) +\|\mu^{1/2}(\nabla\pmb{u}-\nabla_w\pmb{u}_h)\|_{\omega_e} +\|(\mu+\lambda)^{1/2}(\nabla\cdot\pmb{u}-\nabla_w\cdot\pmb{u}_h)\|_{\omega_e}\right).
\end{align}

\end{lemma}
\begin{proof}
Let $\pmb{w}_\tau=(\pmb{f}+\nabla\cdot(\mu\nabla_w\pmb{v}_h)+\nabla((\mu+\lambda)\nabla_w\cdot\pmb{v}_h) )\phi_\tau(\pmb{x})$, where $\phi_\tau(\pmb{x})=27\lambda_1\lambda_2\lambda_3$ is a bubble function defined on $\tau$, we have
\begin{align*}
(\pmb{f}, \pmb{w}_\tau)_\tau=(\mu\nabla\pmb{u},\nabla\pmb{w}_\tau)_\tau+
((\mu+\lambda)\nabla\cdot\pmb{u},\nabla\cdot\pmb{w}_\tau)_\tau.
\end{align*}

Subtracting and adding  $(\pmb{f}_h, \pmb{w}_\tau)_\tau$ , $(\mu\nabla_w\pmb{u}_h,\pmb{w}_\tau)_\tau$ and
$((\mu+\lambda)\nabla_w\cdot\pmb{u}_h,\nabla\cdot\pmb{w}_\tau)_\tau$ from both sides of the above equation,  we get
\begin{align*}
&(\pmb{f}-\pmb{f}_h, \pmb{w}_\tau)_\tau+(\pmb{f}_h, \pmb{w}_\tau)_\tau-(\mu\nabla_w\pmb{u}_h,\nabla\pmb{w}_\tau)_\tau-
((\mu+\lambda)\nabla_w\cdot\pmb{u}_h,\nabla\cdot\pmb{w}_\tau)_\tau\\
&=(\mu(\nabla\pmb{u}- \nabla_w\pmb{u}_h),\nabla\pmb{w}_\tau)_\tau+
((\mu+\lambda)(\nabla\cdot\pmb{u}-\nabla_w\cdot\pmb{u}_h),\nabla\cdot\pmb{w}_\tau)_\tau.
\end{align*}

Using the integration by parts, inverse inequality and $\pmb{w}_\tau|_{\partial \tau}=0$, the above equation becomes
\begin{align*}
&(\pmb{f}_h+\nabla\cdot(\mu\nabla_w\pmb{u}_h)+\nabla((\mu+\lambda)\nabla_w\cdot\pmb{u}_h), \pmb{w}_\tau)_\tau \\
&=(\mu(\nabla\pmb{u}- \nabla_w\pmb{u}_h),\nabla\pmb{w}_\tau)_\tau+
((\mu+\lambda)(\nabla\cdot\pmb{u}-\nabla_w\cdot\pmb{u}_h),\nabla\cdot\pmb{w}_\tau)_\tau-(\pmb{f}-\pmb{f}_h, \pmb{w}_\tau)_\tau.
\end{align*}

Using the properties of the bubble function $\phi_\tau(\pmb{x})$, we obtain
\begin{align*}
&\|\pmb{f}_h+\nabla\cdot(\mu\nabla_w\pmb{u}_h)+\nabla((\mu+\lambda)\nabla_w\cdot\pmb{u}_h)\|^2_\tau \\
&\leq \mu^{1/2} \|\mu^{1/2}(\nabla\pmb{u}- \nabla_w\pmb{u}_h)\|_{\tau}\|\nabla\pmb{w}_\tau\|_\tau+
(\mu+\lambda)^{1/2}\|(\mu+\lambda)^{1/2}(\nabla\cdot\pmb{u}-\nabla_w\cdot\pmb{u}_h)_\tau\|\nabla\cdot\pmb{w}_\tau\|_\tau+\|\pmb{f}-\pmb{f}_h\|_\tau\|\pmb{w}_\tau\|_\tau\\
&\leq C(\|\mu^{1/2}(\nabla\pmb{u}- \nabla_w\pmb{u}_h)\|_{\tau} +
\|(\mu+\lambda)^{1/2}(\nabla\cdot\pmb{u}-\nabla_w\cdot\pmb{u}_h)\|_\tau h_\tau^{-1}\|\pmb{w}_\tau\|_\tau+\|\pmb{f}-\pmb{f}_h\|_\tau\|\pmb{w}_\tau\|_\tau\\
&\leq C(\|\mu^{1/2}(\nabla\pmb{u}- \nabla_w\pmb{u}_h)\|_{\tau} +
\|(\mu+\lambda)^{1/2}(\nabla\cdot\pmb{u}-\nabla_w\cdot\pmb{u}_h)\|_\tau +h_\tau\|\pmb{f}-\pmb{f}_h\|_\tau) h_\tau^{-1}\|\pmb{w}_\tau\|_\tau.
\end{align*}

Notice that $\|\pmb{w}_\tau\|_\tau=\|\pmb{f}_h+\nabla\cdot(\mu\nabla_w\pmb{u}_h)+\nabla((\mu+\lambda)\nabla_w\cdot\pmb{u}_h)\|_\tau$,  we have
\begin{align*}
&h_\tau\|\pmb{f}_h+\nabla\cdot(\mu\nabla_w\pmb{u}_h)+\nabla((\mu+\lambda)\nabla_w\cdot\pmb{u}_h)\|_\tau \\ 
&\leq C(\|\mu^{1/2}(\nabla\pmb{u}- \nabla_w\pmb{u}_h)\|_{\tau} +
\|(\mu+\lambda)^{1/2}(\nabla\cdot\pmb{u}-\nabla_w\cdot\pmb{u}_h)\|_\tau+h_\tau\|\pmb{f}-\pmb{f}_h\|_\tau).  
\end{align*}
Then we  have completed the proof of (\ref{LocalLowerBou:1}) .
 
Let $\pmb{v}_e=[(\mu \nabla_w\pmb{v}_h+(\mu+\lambda)(\nabla_w\cdot\pmb{v}_h )\pmb{I})\pmb{n}]_e\phi_e(\pmb{x})$,  where $\phi_e(\pmb{x})$ is the bubble function defined on the edge or face   $e$,  we arrive at  
\begin{align*}
\sum\limits_{\tau\in \omega_e}(\pmb{f},\pmb{v}_e)_\tau=\sum\limits_{\tau\in \omega_e}(\mu\nabla\pmb{u},\nabla\pmb{v}_e)_\tau+\sum\limits_{\tau\in \omega_e}((\mu+\lambda)\nabla\cdot\pmb{u},\nabla\cdot\pmb{v}_e)_\tau.
\end{align*}

Subtracting $\sum\limits_{\tau\in \omega_e}(\mu\nabla_w\pmb{u}_h,\nabla\pmb{v}_e)_\tau$ and $\sum\limits_{\tau\in \omega_e}((\mu+\lambda)\nabla_w\cdot\pmb{u}_h,\nabla\cdot\pmb{v}_e)_\tau$ from both sides of the above equation
, then  we get
\begin{align*}
&\sum\limits_{\tau\in \omega_e}(\pmb{f},\pmb{v}_e)_\tau-\sum\limits_{\tau\in \omega_e}(\mu\nabla_w\pmb{u}_h,\nabla\pmb{v}_e)_\tau-\sum\limits_{\tau\in \omega_e}((\mu+\lambda)\nabla_w\cdot\pmb{u}_h,\nabla\cdot\pmb{v}_e)_\tau\\
&=\sum\limits_{\tau\in \omega_e}(\mu(\nabla\pmb{u}-\nabla_w\pmb{u}_h),\nabla\pmb{v}_e)_\tau+\sum\limits_{\tau\in \omega_e}((\mu+\lambda)(\nabla\cdot\pmb{u}-\nabla_w\cdot\pmb{u}_h),\nabla\cdot\pmb{v}_e)_\tau.
\end{align*}

Using the properties of the bubble function $\phi_\tau(\pmb{x})$ and the integration by parts, we obtain
\begin{align*}
&\sum\limits_{\tau\in \omega_e}(\pmb{f},\pmb{v}_e)_\tau+\sum\limits_{\tau\in \omega_e}(\nabla\cdot(\mu\nabla_w\pmb{u}_h),\pmb{v}_e)_\tau-  <\mu\nabla_w\pmb{u}_h\pmb{n}_1,\pmb{v}_e>_e-  <\mu\nabla_w\pmb{u}_h\pmb{n}_2,\pmb{v}_e>_e\\
&+\sum\limits_{\tau\in \omega_e}(\nabla((\mu+\lambda)\nabla_w\cdot\pmb{u}_h),\pmb{v}_e)_\tau-
<(\mu+\lambda)\nabla_w\cdot\pmb{u}_h,\pmb{v}_e\cdot\pmb{n}_1>_e-<(\mu+\lambda)\nabla_w\cdot\pmb{u}_h,\pmb{v}_e\cdot\pmb{n}_2>_e\\
&=\sum\limits_{\tau\in \omega_e}(\mu(\nabla\pmb{u}-\nabla_w\pmb{u}_h),\nabla\pmb{v}_e)_\tau+\sum\limits_{\tau\in \omega_e}((\mu+\lambda)(\nabla\cdot\pmb{u}-\nabla_w\cdot\pmb{u}_h),\nabla\cdot\pmb{v}_e)_\tau.
\end{align*}

Merging the inner product of the edges, we have
\begin{align*}
&\sum\limits_{\tau\in \omega_e}(\pmb{f},\pmb{v}_e)_\tau+\sum\limits_{\tau\in \omega_e}(\nabla\cdot(\mu\nabla_w\pmb{u}_h),\pmb{v}_e) _\tau+\sum\limits_{\tau\in \omega_e}(\nabla((\mu+\lambda)\nabla_w\cdot\pmb{u}_h),\pmb{v}_e)\\
&-  <(\mu\nabla_w\pmb{u}_h+(\mu+\lambda)(\nabla_w\cdot\pmb{u}_h)\pmb{I})\pmb{n}_1,\pmb{v}_e>_e -  <(\mu\nabla_w\pmb{u}_h+(\mu+\lambda)(\nabla_w\cdot\pmb{u}_h)\pmb{I})\pmb{n}_2,\pmb{v}_e>_e\\
&=\sum\limits_{\tau\in \omega_e}(\mu(\nabla\pmb{u}-\nabla_w\pmb{u}_h),\nabla\pmb{v}_e)_\tau+\sum\limits_{\tau\in \omega_e}((\mu+\lambda)(\nabla\cdot\pmb{u}-\nabla_w\cdot\pmb{u}_h),\nabla\cdot\pmb{v}_e)_\tau.
\end{align*}

Subtracting and adding $(\pmb{f}_h,\pmb{v}_e)_\tau$ and merging, we obtain
\begin{align*}
&<[(\mu\nabla_w\pmb{u}_h+(\mu+\lambda)(\nabla_w\cdot\pmb{u}_h)\pmb{I})\pmb{n}],\pmb{v}_e>_e \\
&=\sum\limits_{\tau\in \omega_e}(\pmb{f}-\pmb{f}_h,\pmb{v}_e)_\tau+\sum\limits_{\tau\in \omega_e}(\pmb{f}_h+\nabla\cdot(\mu\nabla_w\pmb{u}_h)+\nabla((\mu+\lambda)\nabla_w\cdot\pmb{u}_h),\pmb{v}_e)_\tau\\
&~~~~-\sum\limits_{\tau\in \omega_e}(\mu(\nabla\pmb{u}-\nabla_w\pmb{u}_h),\nabla\pmb{v}_e)_\tau-\sum\limits_{\tau\in \omega_e}((\mu+\lambda)(\nabla\cdot\pmb{u}-\nabla_w\cdot\pmb{u}_h),\nabla\cdot\pmb{v}_e)_\tau.
\end{align*}

Using  Schwarz inequality and  inverse inequality, the above equation becomes
\begin{align*}
&\|[(\mu \nabla_w\pmb{v}_h+(\mu+\lambda)(\nabla_w\cdot\pmb{v}_h )\pmb{I})\pmb{n}]\|^2_e\\
&\leq \left|
\sum\limits_{\tau\in \omega_e}(\pmb{f}-\pmb{f}_h,\pmb{v}_e)_\tau+\sum\limits_{\tau\in \omega_e}(\pmb{f}_h+\nabla\cdot(\mu\nabla_w\pmb{u}_h)+\nabla((\mu+\lambda)\nabla_w\cdot\pmb{u}_h),\pmb{v}_e)_\tau\right|\\
&~~~~+\left|\sum\limits_{\tau\in \omega_e}(\mu(\nabla\pmb{u}-\nabla_w\pmb{u}_h),\nabla\pmb{v}_e)_\tau+\sum\limits_{\tau\in \omega_e}((\mu+\lambda)(\nabla\cdot\pmb{u}-\nabla_w\cdot\pmb{u}_h),\nabla\cdot\pmb{v}_e)_\tau\right|\\
&\lesssim  
 \left(\|\pmb{f}-\pmb{f}_h\|_{\omega_e}+ \| \pmb{f}_h+ \nabla\cdot(\mu\nabla_w\pmb{u}_h)+\nabla((\mu+\lambda)\nabla_w\cdot\pmb{u}_h\|_{\omega_e}\right)  \|\pmb{v}_e\|_{\omega_e}\\
&~~~~+\|\mu^{1/2}(\nabla\pmb{u}-\nabla_w\pmb{u}_h)\|_{\omega_e}\|\nabla\pmb{v}_e\|_{\omega_e}+\|(\mu+\lambda)^{1/2}(\nabla\cdot\pmb{u}-\nabla_w\cdot\pmb{u}_h)\|_{\omega_e} \|\nabla\cdot\pmb{v}_e\|_{\omega_e}\\
&\lesssim
 \left(\|\pmb{f}-\pmb{f}_h\|_{\omega_e}+ \| \pmb{f}_h+ \nabla\cdot(\mu\nabla_w\pmb{u}_h)+\nabla((\mu+\lambda)\nabla_w\cdot\pmb{u}_h\|_{\omega_e}\right)  \|\pmb{v}_e\|_{\omega_e}\\
&~~~~+\|\mu^{1/2}(\nabla\pmb{u}-\nabla_w\pmb{u}_h)\|_{\omega_e} h_e^{-1}\|\pmb{v}_e\|_{\omega_e}+\|(\mu+\lambda)^{1/2}(\nabla\cdot\pmb{u}-\nabla_w\cdot\pmb{u}_h)\|_{\omega_e} h_e^{-1}\|\pmb{v}_e\|_{\omega_e}\\
&\lesssim
 \left(h_e^{1/2}\|\pmb{f}-\pmb{f}_h\|_{\omega_e}+h_e^{1/2} \| \pmb{f}_h+ \nabla\cdot(\mu\nabla_w\pmb{u}_h)+\nabla((\mu+\lambda)\nabla_w\cdot\pmb{u}_h\|_{\omega_e}\right.\\
&~~~~\left.+h_e^{-1/2}\|\mu^{1/2}(\nabla\pmb{u}-\nabla_w\pmb{u}_h)\|_{\omega_e} +h_e^{-1/2}\|(\mu+\lambda)^{1/2}(\nabla\cdot\pmb{u}-\nabla_w\cdot\pmb{u}_h)\|_{\omega_e}\right)\|\pmb{v}_e\|_{e}\\
&\lesssim
 \left(h_e^{1/2}\|\pmb{f}-\pmb{f}_h\|_{\omega_e}+h_e^{1/2} \| \pmb{f}_h+ \nabla\cdot(\mu\nabla_w\pmb{u}_h)+\nabla((\mu+\lambda)\nabla_w\cdot\pmb{u}_h\|_{\omega_e}\right.\\
&~~~~\left.+h_e^{-1/2}\|\mu^{1/2}(\nabla\pmb{u}-\nabla_w\pmb{u}_h)\|_{\omega_e} +h_e^{-1/2}\|(\mu+\lambda)^{1/2}(\nabla\cdot\pmb{u}-\nabla_w\cdot\pmb{u}_h)\|_{\omega_e}\right)  \\
&~~~~\|(\mu \nabla_w\pmb{v}_h+(\mu+\lambda)(\nabla_w\cdot\pmb{v}_h )\pmb{I})\pmb{n}\|_e.
\end{align*}

Dividing   $\|(\mu \nabla_w\pmb{v}_h+(\mu+\lambda)(\nabla_w\cdot\pmb{v}_h )\pmb{I})\pmb{n}\|^2_e$ on both sides of the above inequality,  and using (\ref{LocalLowerBou:1}) and the definition $\mathrm{osc}(\pmb{f}, \omega_e)$, we obtain
\begin{align*}
&h_e^{1/2}\|[(\mu \nabla_w\pmb{v}_h+(\mu+\lambda)(\nabla_w\cdot\pmb{v}_h )\pmb{I})\pmb{n}]\|_e\\
&\leq C 
 \left(\mathrm{osc}(\pmb{f}, \omega_e) +\|\mu^{1/2}(\nabla\pmb{u}-\nabla_w\pmb{u}_h)\|_{\omega_e} +\|(\mu+\lambda)^{1/2}(\nabla\cdot\pmb{u}-\nabla_w\cdot\pmb{u}_h)\|_{\omega_e}\right).
\end{align*}
Then we  have completed the proof of (\ref{LocalLowerBou:2}) .
\end{proof}

Using  (\ref{LocalLowerBou:1}) and (\ref{LocalLowerBou:2}), then summing over all $e\in \mathcal{E}_h$ and all $\tau\in\mathcal{T}_h$,  we arrive at the following lower bound for the error estimator.
 \begin{theorem}[Lower Bound]
 Let  $\pmb{u}$  be the solution of (\ref{ela-Model:1}) and  $\pmb{u}_h=\{\pmb{u}_0^h,\pmb{u}_b^h\}\in  \mathcal{V}_h$ be the solution of (\ref{Dis-Weak-Gal}), respectively.  There exists a constant $C_1>0$, such that such 
\begin{align*}
\eta^2(\pmb{u}_h, \mathcal{T}_h)\leq C_2\left(|\mu^{1/2}( \nabla\pmb{u}- \nabla_{w}\pmb{u}_h)\|^2_{\tau}+
\|(\mu+\lambda)^{1/2} (\nabla\cdot\pmb{u}- \nabla_{w}\cdot\pmb{u}_h)\|^2_{\tau}+ \mathrm{osc}^2(\pmb{f},\mathcal{T}_h) \right).
\end{align*} 
wher the constant $C_2>0$, only depends on the shape regularity of $\mathcal{T}_h$ .
 \end{theorem}

 \section{Numeriacl Experiments}
\setcounter{equation}{0}

In this section,  we  give two experiments to verify the theoretical result.  During these experiments, we adopt the lowest order($k=1$) during  the weak finite element space $\mathcal{T}_h$ and the energy norm $\|\cdot\|_A=a_w(\cdot,\cdot)^{\frac{1}{2}}$ to do the error analysis,  the Lam\'{e} constants $\mu=0.5$ and    
$\lambda=1.0$, and the error control constant $tol = 10^{-8}$.

\begin{example}\label{example1}
In this example, we examine the 'L-shape' problem in two dimension. Let $\Omega =(-1,1)^2\backslash (0,1)\times(-1,0)$, the proper vector source $\pmb{f}$ and the boundary function $\pmb{g}$ are chosen to ensure the solution 
\begin{align*}
\pmb{u}= (u_1, u_2)^T,
\end{align*}
where 
$u_1(x, y) = u_2 (x, y)= r^{\frac{2}{3}}\sin(\frac{2}{3}\theta)$ in polar coordinates.
\end{example}

During this example, we adopt the initial mesh like the left figure of  Figure \ref{figure-example-1}.  After performing the AMWG-FEM, we can see that the refinement elements are concentrated with singular of the solution $\pmb{u}$.  The right figure of Figure   \ref{figure-example-1} shows  the 11st   refinement meshes with $\vartheta=0.5$. 

In the left figure of Figure \ref{figure-example-2}, the abscissa value  represents the number of unknowns of the mesh $\mathcal{T}_l$, and the ordinate value $\|\pmb{u}-\pmb{u}_l\|_{A}$  represents the energy norm of the error between the solution $\pmb{u}$ and the modified weak Galerkin finite element solution $\pmb{u}_l$.     we present the error curve  about the errors of the solution $\pmb{u}$ and the modified weak Galerkin finite element solution $\pmb{u}_l$  under the energy norm with the D\''{o}ffler parameter  $\vartheta = 0.1, 0.3, 0.5$ .  In the right of  of Figure \ref{figure-example-2},  the ordinate value represents    the estimator  $\eta(\pmb{u}_l, \mathcal{T}_l)$.  All straight lines with slope $\frac{1}{2}$ can be moved vertically in Figure \ref{figure-example-2}.
\begin{figure}[!ht]
\begin{minipage}[t]{0.52\linewidth}
\centering
\includegraphics[width=3in]{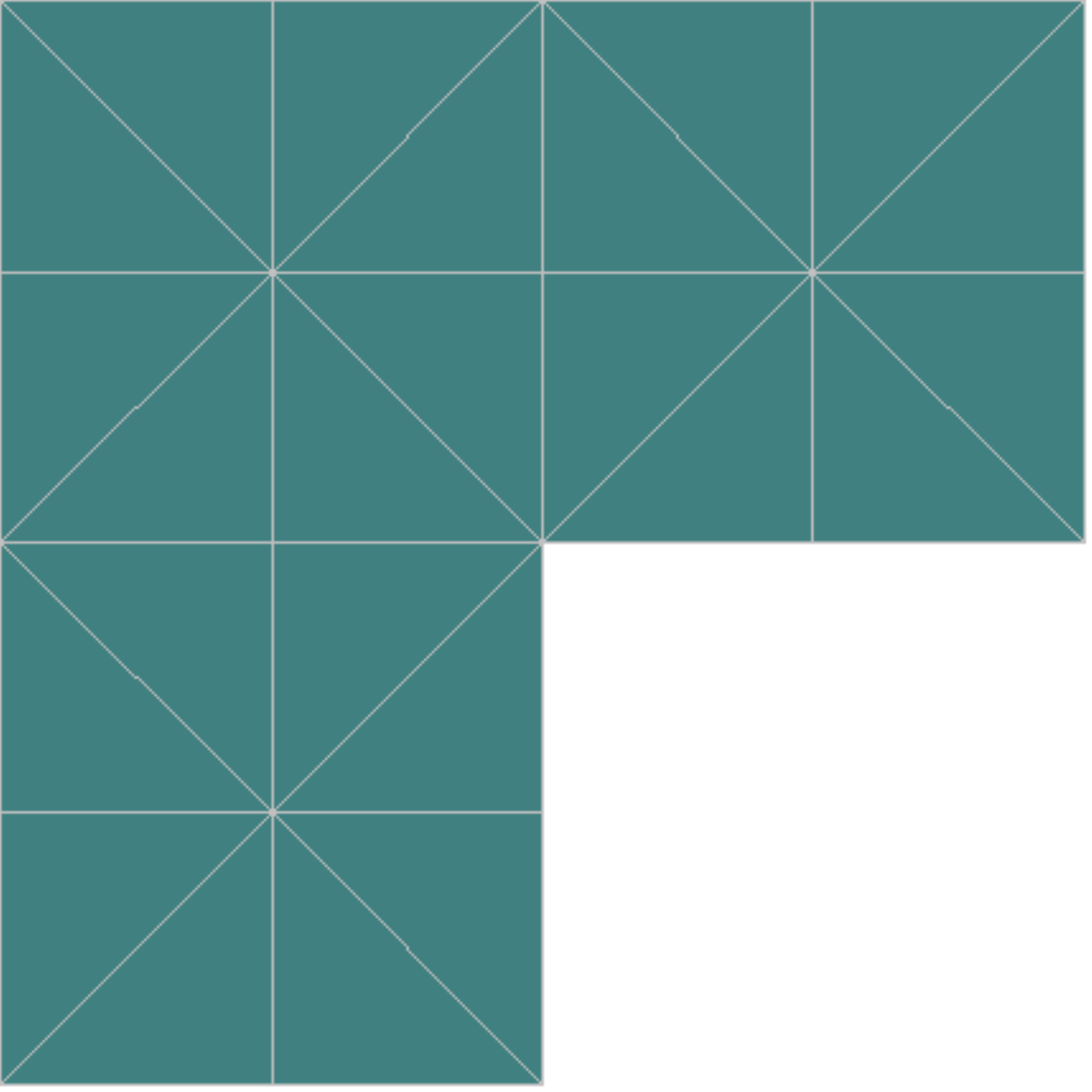}
\end{minipage}%
\begin{minipage}[t]{0.52\linewidth}
\centering
\includegraphics[width=3in]{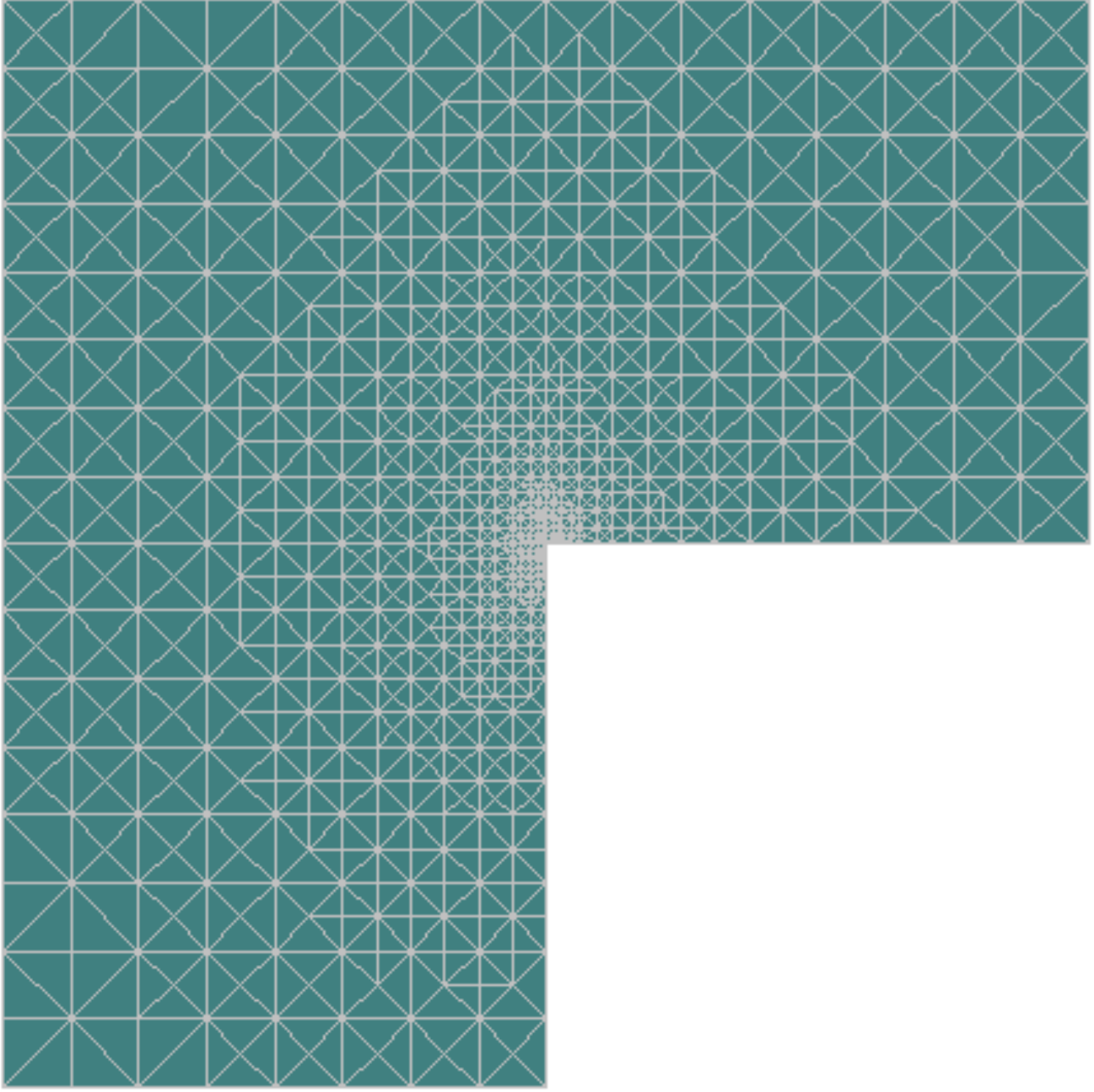}
\end{minipage}
\caption{The initial mesh (left) and the 11st refinement mesh(right) of Exmaple \ref{example1}.} \label{figure-example-1}
\end{figure}

\begin{figure}[!ht]
\begin{minipage}[t]{0.52\linewidth}
\centering
\includegraphics[width=3in]{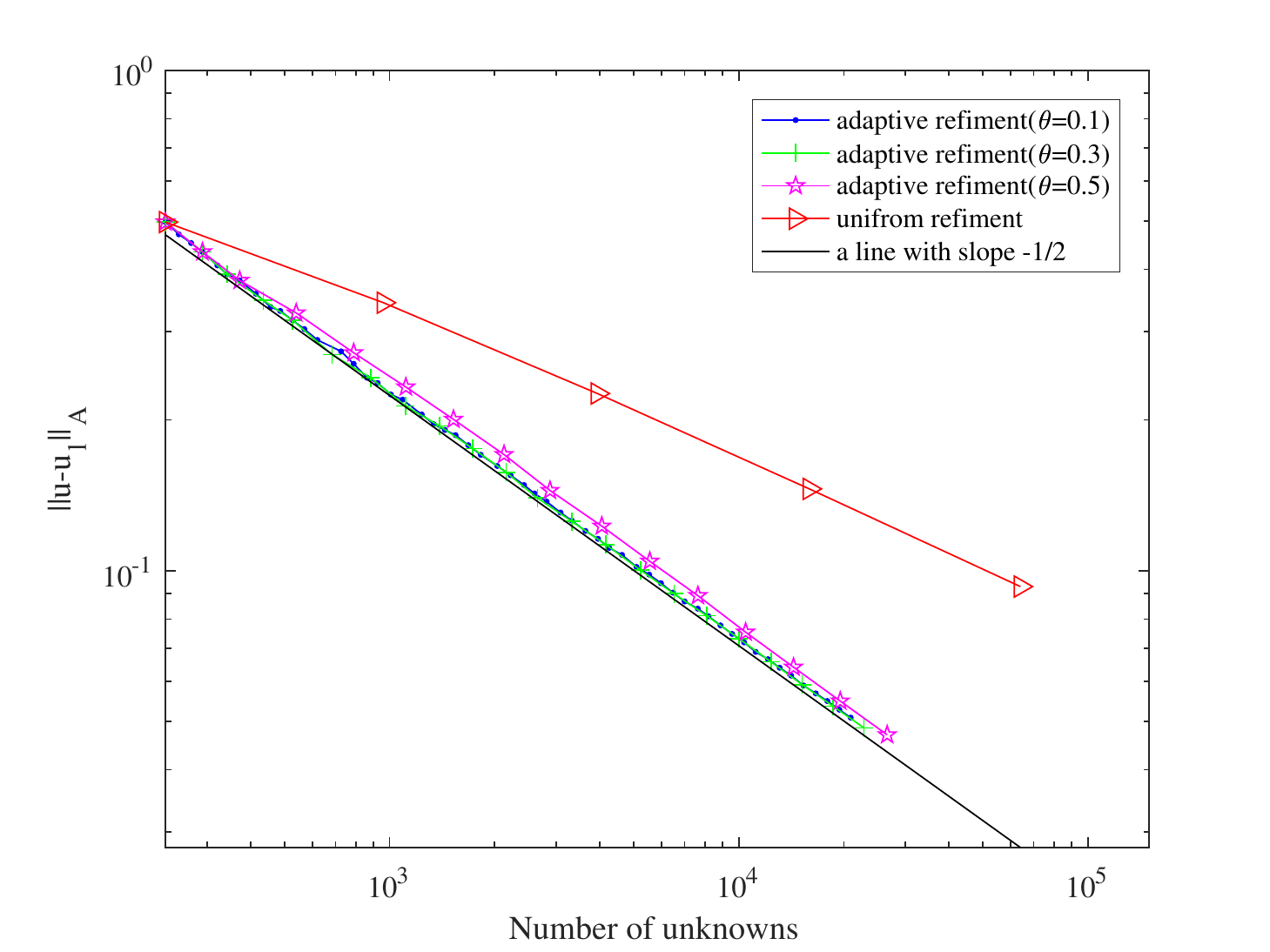}
\end{minipage}%
\begin{minipage}[t]{0.52\linewidth}
\centering
\includegraphics[width=3in]{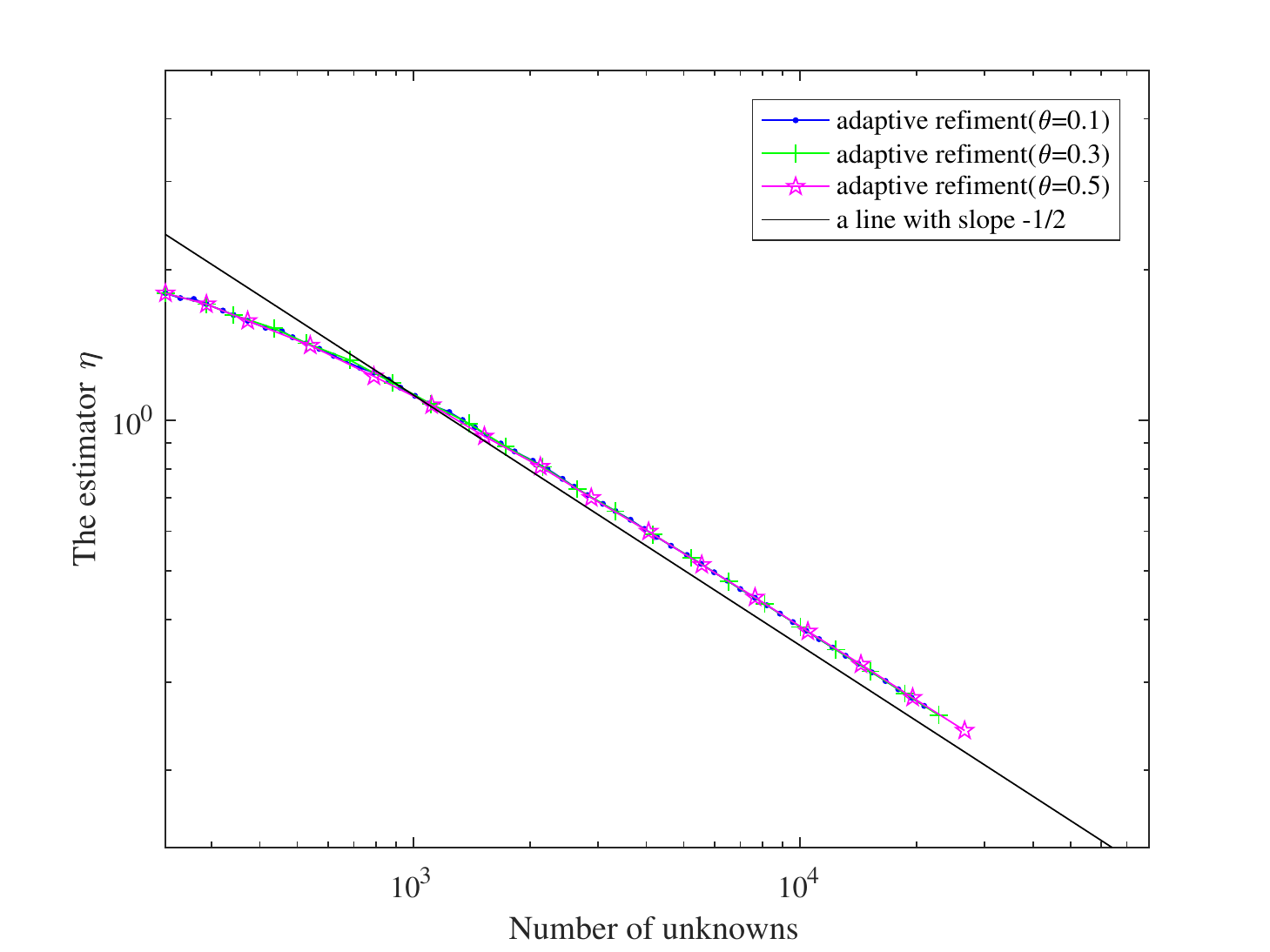}
\end{minipage}
\caption{The curves of  $\|\pmb{u}-\pmb{u}_l\| _A$ (left) and $\eta(\pmb{u}_l, \mathcal{T}_l)$ (right) for $\vartheta=0.1, 0.3, 0.5$  of Exmaple \ref{example1}.} \label{figure-example-2}
\end{figure}

\begin{example}\label{example2}
In this example, we examine the 'L-shape' problem in three dimension. Let $\Omega =(-1,1)^3\backslash (0,1)\times(0,1)\times(-1,1)$, the proper vector source $\pmb{f}$ and the boundary function $\pmb{g}$ are chosen to ensure the solution  
\begin{align*}
\pmb{u}= (u_1, u_2, u_3)^T,
\end{align*}
where 
$u_1(x, y, z) = u_2 (x, y, z)= u_3 (x, y, z)= r^{\frac{2}{3}}\sin(\frac{2}{3}\theta)$, in polar coordinates.

\end{example}

During this example, we adopt the initial mesh like the left figure of  Figure \ref{figure-example-3}.  After performing the AMWG-FEM, we can see that the refinement elements are concentrated with singular of the solution $\pmb{u}$.  The right figure of Figure   \ref{figure-example-3} shows  the 12nd  refinement meshes with $\vartheta=0.5$. 

In the left figure of Figure \ref{figure-example-4}, the abscissa value  represents the number of unknowns of the mesh $\mathcal{T}_l$, and the ordinate value $\|\pmb{u}-\pmb{u}_l\|_{A}$  represents the energy norm of the error between the solution $\pmb{u}$ and the modified weak Galerkin finite element solution $\pmb{u}_l$.     we present the error curve  about the errors of the solution $\pmb{u}$ and the modified weak Galerkin finite element solution $\pmb{u}_l$  under the energy norm with the D\''{o}ffler parameter  $\vartheta = 0.1, 0.3, 0.5$ .  In the right of  of Figure \ref{figure-example-4},  the ordinate value represents    the estimator  $\eta(\pmb{u}_l, \mathcal{T}_l)$.  All straight lines with slope $\frac{1}{3}$ can be moved vertically in Figure \ref{figure-example-4}.

\begin{figure}[!ht]
\begin{minipage}[t]{0.52\linewidth}
\centering
\includegraphics[width=3in]{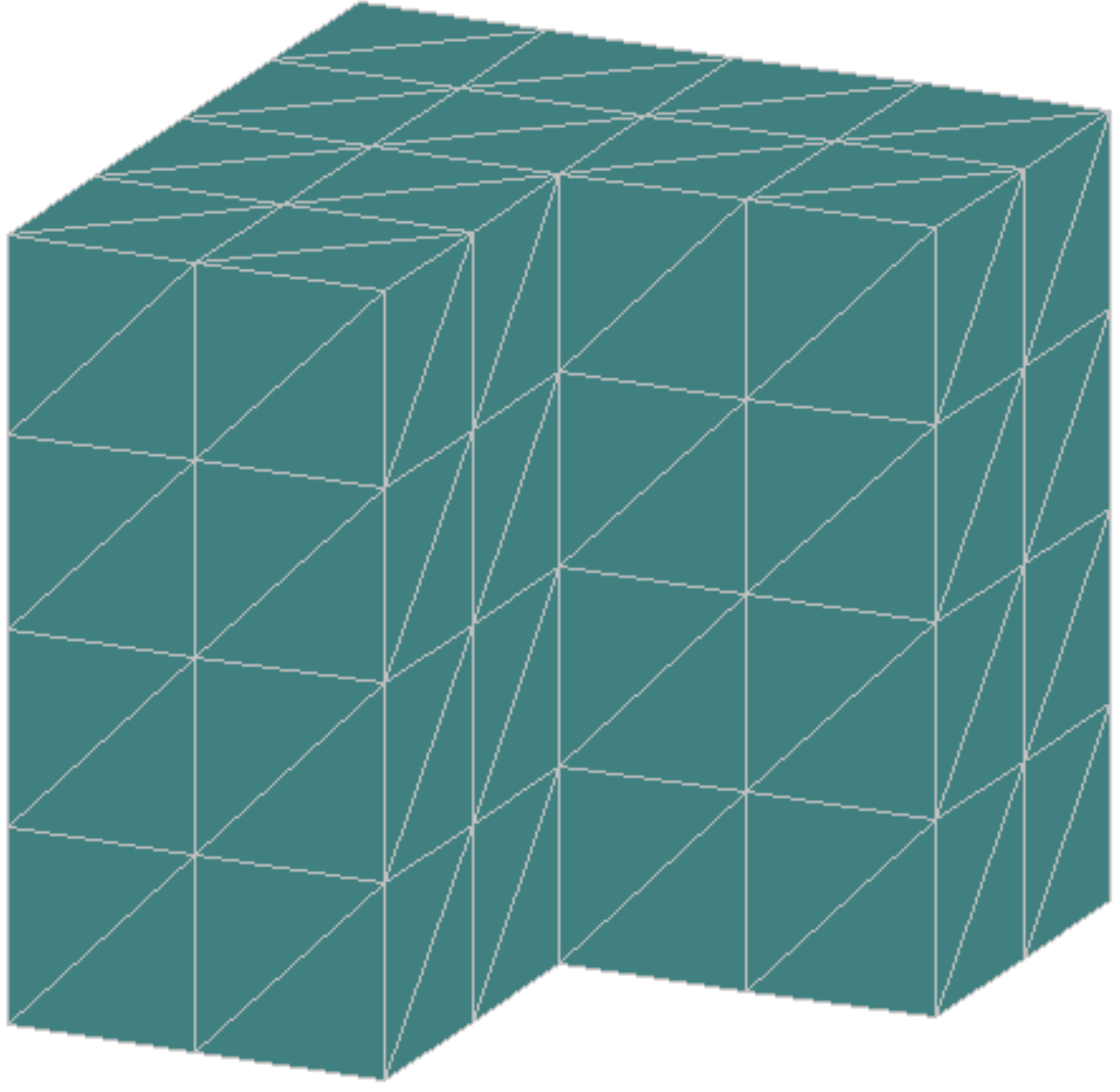}
\end{minipage}%
\begin{minipage}[t]{0.52\linewidth}
\centering
\includegraphics[width=3in]{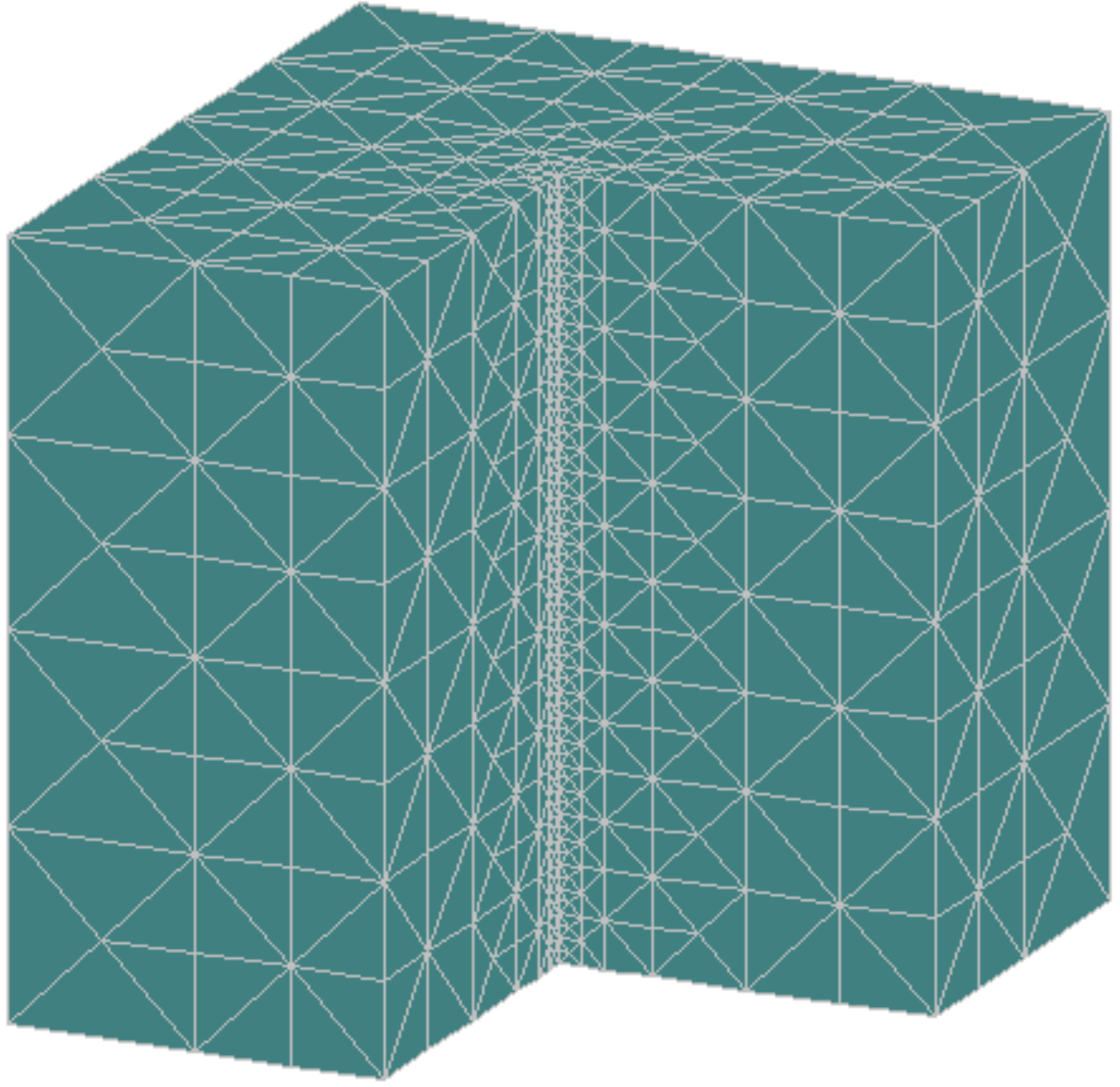}
\end{minipage}
\caption{The initial mesh (left) and the 12nd refinement mesh(right) of Exmaple \ref{example2}.}\label{figure-example-3}
\end{figure}

\begin{figure}[!ht]
\begin{minipage}[t]{0.52\linewidth}
\centering
\includegraphics[width=3in]{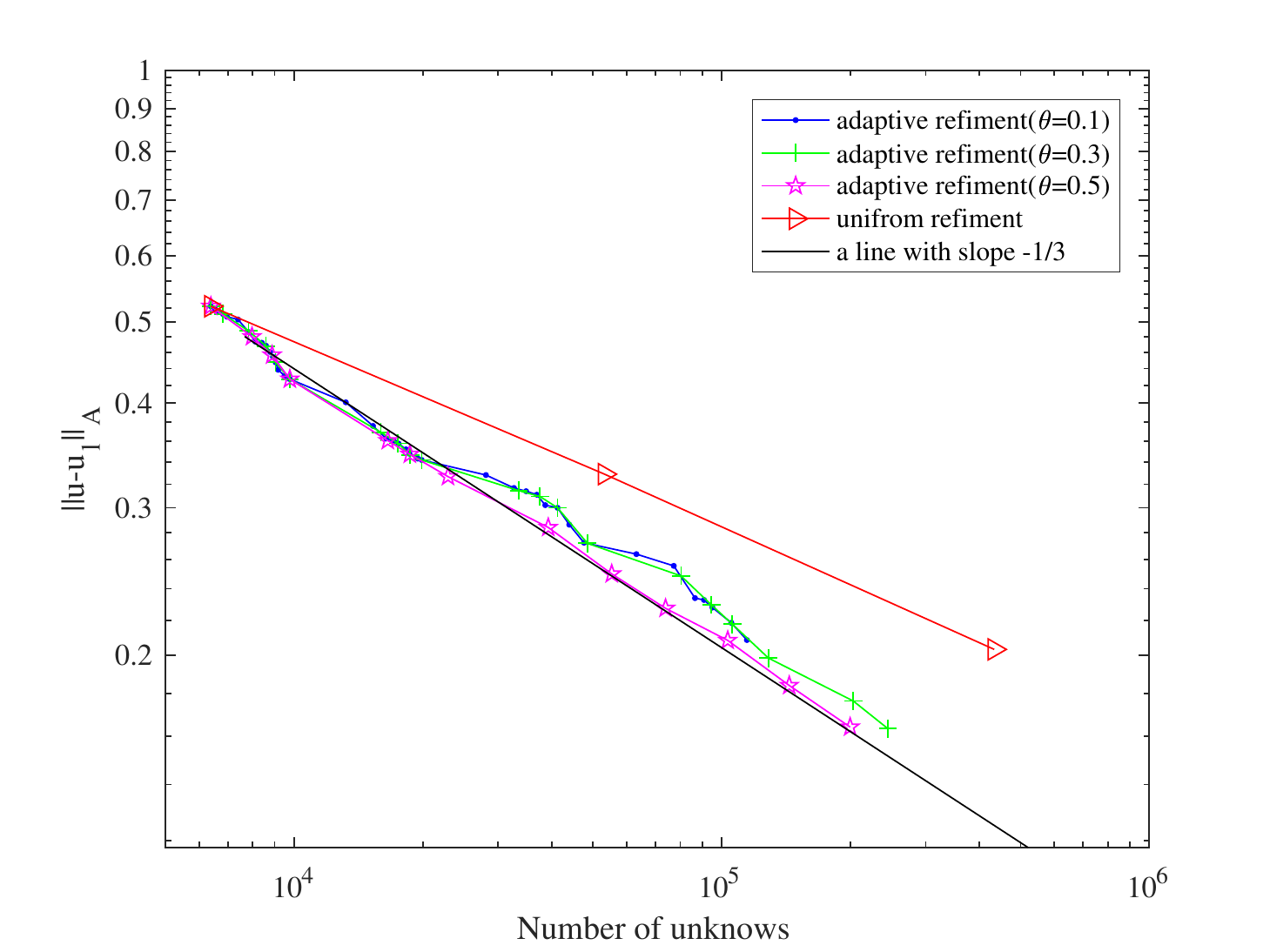}
\end{minipage}%
\begin{minipage}[t]{0.52\linewidth}
\centering
\includegraphics[width=3in]{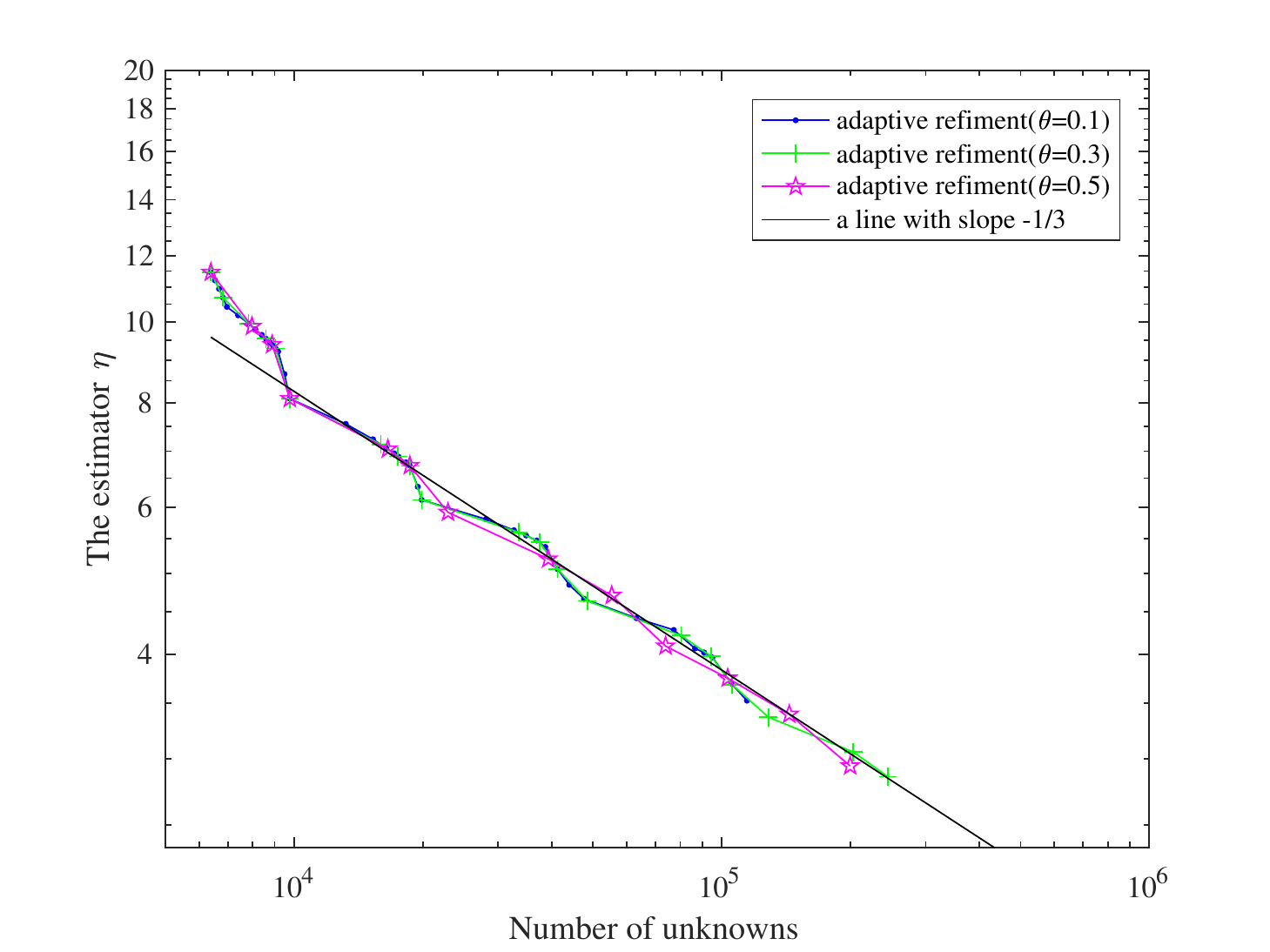}
\end{minipage}
\caption{The curves of  $\|\pmb{u}-\pmb{u}_l\| _A$ (left) and $\eta(\pmb{u}_l, \mathcal{T}_l)$ (right) for $\vartheta=0.1, 0.3, 0.5$  of Exmaple \ref{example2}.} \label{figure-example-4}
\end{figure}

 From the above numerical examples, we verify the relibility and efficiency of the estimator in (\ref{eta:1}), and the convergence of the AMWG-FEM.



\begin{thebibliography}{10}

 \bibitem{Adler2019method}
J. Adler, X. Hu, L. Mu, X. Ye.  An a posteriori error estimator for the weak Galerkin
least-squares finite-element method. \emph{J. Comput. Appl. Math.}, 362:383-399, 2019.

\bibitem{Bao2019meshes}
F. Bao, L. Mu, J. Wang. A fully computable a posteriori error estimate for
the Stokes equations on polytopal meshes,  SIAM J. Numer. Anal., 57(1): 458-477, 2019.

\bibitem{BonitoNochetto10:734}
 A. Bonito and R.H. Nochetto. Quasi-optimal convergence rate of an adaptive
discontinuous galerkin method. \emph{SIAM J. Numer. Anal.}, 48(2):734–771, 2010.

\bibitem{Brenner1992elasticity}
S.C.Brenner, L.Y. Sung. Linear finite element methods for planar linear elasticity. \emph{Math.Comp.},57:321-338, 1992.

\bibitem{Brenner1994Methods}
S.C.Brenner, L.R.Scott. The mathematical theory of finite element methods, in: Texts in Applied Mathematics, vol. 15, Spring-Verlag, New York, 1994.

\bibitem{Cascon2008method}
M. Cascon, C. Kreuzer,  H. Nochetto, G.  Siebert. Quasi-optimal convergence rate
for an adaptive finite element method.  \emph{SIAM J. Numer.
Anal.}, 46:2524-2550, 2008.

\bibitem{Chen2014problems}
L. Chen, J. Wang, X.Ye. A posteriori error estimates for weak Galerkin finite element methods for
second order elliptic problems. \emph{J. Sci. Comput.}, 59(2): 496-51, 2014.

\bibitem{Chen2016stresses} 
G. Chen, X. Xie. A robust weak Galerkin finite element method for linear elasticity with strong symmeric stresses. \emph{Comput. Methods Appl. Math.}, 16: 389-408, 2016.


\bibitem{Ciarlet1972Problems}
 P.G. Ciarlet, The Finite Element Method for Elliptic Problems. \emph{North-Holland}, Amsterdam, 1978.

\bibitem{Dorfler1966equation}
W. D\''{o}rfler. A convergent adaptive algorithm for Poisson's equation. \emph{SIAM J. Numer.
Anal.}, 33(3):1106–1124, 1996.

\bibitem{Harper2019meshes}
G. Harper, J. Liu, S. Tavener, B. Zheng. Lowest-order weak Galerkin finite element methods for linear elasticity on rectangular and brick meshes. \emph{J. Sci. Comput.}, 78:1917-1941, 2019. 


\bibitem{Huang2017equations} 
Y. Huang, J. Li, D. Li.  Developing weak Galerkin finite element methods for the wave equation. Numer. Methods Partial Differential Equations, 33(3): 868-884, 2017.
 

\bibitem{Hu2019equations} 
X. Hu, L. Mu, X. Ye. A weak Galerkin finite element method for the Navier- Stokes equation. \emph{J. Comput. Appl. Math.}, 362:614-625, 2019.

\bibitem{Li2018problems}
G. Li, Y. Chen, Y. Huang. A new Galerkin finite element scheme for general second-order elliptic problems. \emph{J. Comput. Appl. Math.}, 344: 701-715, 2018. 

\bibitem{Lihengguang2019meshes} 
H. Li, L. Mu, X.Ye. A posteriori error estimates for the weak Galerkin finite element methods on polytopal meshes. \emph{Commun. Comput. Phys.}, 26(2):558-578, 2019.

\bibitem{Liu2018equations}
X. Liu, J. Li, Z. Chen. A weak Galerkin finite element method for the Navier-Stokes equations.  \emph{J. Comput. Appl. Math.}, 333: 442-457, 2018.
 
\bibitem{Liu2016dimensions}
C. Liu, L. Zhong. S. Shu, Y. Xiong.
Quasi-optimal complexity of adaptive finite element method
for linear elasticity problems in two dimensions. \emph{App. Math. Mech. -Engl. Ed.}, 37(2):151-168, 2016.

\bibitem{Mu2014equation} 
L. Mu, J. Wang, X.Ye. A $C^0$-weak Galerkin finite element method for the biharmonic equation. \emph{J. Sci. Comput.}, 59:473-495, 2014.

\bibitem{Wang2013problems}
J. Wang, X. Ye. A weak Galerkin finite element method for second-order elliptic problems. \emph{J. Comput. Appl. Math.}, 241: 103-115, 2013.

\bibitem{Wang2016equtions}
J. Wang,  X. Ye. A weak Galerkin finite element method for the stokes equations. \emph{Adv. Comput. Math.}, 42:155-174, 2016. 
 
\bibitem{Wang2016stokesequations}
R. Wang, X. Wang, R. Zhang. A weak Galerkin finite element scheme for solving the stationary Stokes equations.  \emph{J. Comput. Appl. Math.}, 302:171-185 2016. 

\bibitem{WangZhai2018problem} 
X. Wang, Q. Zhai, R. Wang, R. Jari. A absolutely stable weak Galerkin finite element method for the Darcy-Stokes problem. \emph{Appl. Math. Comput.}, 331:20-32, 2018.
 
\bibitem{Wang2016formulation}
C. Wang, J. Wang, R. Wang, R. Zhang. A locking-free weak Galerkin finite element method for elasticity problems in the primal formulation.  \emph{J. Comput. Appl. Math.}, 307:346-366, 2016.
 
\bibitem{Wang2018form}
R. Wang, X. Wang, K. Zhang, Q. Zhou.  Hybridized weak Galerkin finite element method for  linear elasticity problem in mixed form. \emph{Front. Math. China}, 13: 1121-1140, 2018. 

\bibitem{Wang2018mixedform} 
R. Wang, R. Zhang. A weak Galerkin finite element method for the linear elasticity problem in mixed form. \emph{J. Comp. Math.}, 36(4): 469-491, 2018.  

\bibitem{Wanghui2022mesh}
H. Wang, S. Xu, X. He. A posteriori error estimates of edge residual-type of  weak Galerkin mixed FEM solving second-order elliptic . \emph{Int. J. Comput. Methods.},  19, 2022.  

\bibitem{Xie2021problems}
Y. Xie, L.Zhong.   Convergence of adaptive weak Galerkin finite element
methods for second order elliptic problems. \emph{J. Sci. Comput. },  86:
1-17, 2021.

\bibitem{Xie2022problem}
Y. Xie, L. Zhong, Y. Zeng. Convergence of an adaptive modified WG method
for second-order elliptic problem. \emph{Numerical. Algorithms}, 90:789-808, 2022.

\bibitem{Yi2019elasticity} 
S. Yi.   A lowest-order weak Galerkin method for linear elasticity.\emph{J. Comput. Appl. Math.},  350: 286-298, 2019.

\bibitem{Zhang2015order}
R. Zhang, Q. Zhai. A weak Galerkin finite element scheme for the biharmonic equations by using poynomials of reduced order.  \emph{J. Sci. Comput.},64:559-585, 2015.

\bibitem{Zhang2016problems}
T. Zhang, T. Lin. A posteriori error estimate for a modified weak
Galerkin method solving elliptic problems. \emph{Numer. Methods Partial Differ. Equ.},  33:381-398, 2017.

\bibitem{Zhang2018problems}
T. Zhang, Y. Chen. A posteriori error analysis for the weak Galerkin
method for solving elliptic problems.\emph{Int. J. Comput. Methods.},15, 2018. 

\bibitem{Zhangran2020problems}
R. Zhang.  Weak Galerkin finite element method for linear elasticity problem. \emph{Mathematic Numerica Sinica}, 42(1):1-17,2020.

\bibitem{Zheng2017stokesproblem} 
X. Zheng, X. Xie. A posteriori error estimator for a weak Galerkin finite element solution of the Stokes problem.  \emph{East Asian J. Appl. Math.}, 7(3): 508-529, 2017.

\bibitem{Zhou2019problems}
S. Zhou, F. Gao, B. Li, Z. Sun. Weak Galerkin finite element method with second-order accuracy in time for parabolic problems. \emph{Appl. Math. Lett.}, 90:118-123, 2019. 
 
 
 \bibitem{Zhu2019equation}
H. Zhu, Y. Zou, S. Chai, C. Zhou.  A weak Galerkin method with RT elements for a stockastic parabolic differential equation. \emph{East Asian J. Appl. Math.}, 9(4): 818-830, 2019.  
 \end{thebibliography}
\end{document}